\documentclass[11pt,a4paper]{article}
\usepackage{amssymb}
\usepackage{amsmath}
\usepackage{amsfonts}
\usepackage{bbm}
\usepackage{amsthm}
\usepackage{mathrsfs}
\usepackage{hyperref}
\usepackage{color}
\usepackage[margin=2.41cm]{geometry}
\usepackage[all,cmtip]{xy}
\usepackage[utf8]{inputenc}
\usepackage{graphicx}
\usepackage{varwidth}
\usepackage{comment}
\usepackage{slashed}

\usepackage{upgreek}
\usepackage{rotating}

\usepackage[shortlabels]{enumitem}

\definecolor{darkred}{rgb}{0.8,0.1,0.1}
\hypersetup{
     colorlinks=false,         
     linkcolor=darkred,
     citecolor=blue,
}

\theoremstyle{plain}
\newtheorem{theo}{Theorem}[section]
\newtheorem{lem}[theo]{Lemma}
\newtheorem{propo}[theo]{Proposition}

\theoremstyle{definition}
\newtheorem{defi}[theo]{Definition}

\newtheorem{assu}[theo]{Assumption}

\newtheorem{exx}[theo]{Example}
\newenvironment{ex}
  {\pushQED{\qed}\exx}
  {\popQED\endexx}

\newtheorem{remm}[theo]{Remark}
\newenvironment{rem}
  {\pushQED{\qed}\remm}
  {\popQED\endremm}

\numberwithin{equation}{section}

\def\nn{\nonumber}

\def\bbR{\mathbb{R}}
\def\bbC{\mathbb{C}}
\def\bbN{\mathbb{N}}
\def\bbZ{\mathbb{Z}}
\def\bbT{\mathbb{T}}
\def\bbS{\mathbb{S}}

\def\sp{\mathrm{sp}}

\def\id{\mathrm{id}}

\def\dd{\mathrm{d}}

\def\Mod{\mathsf{Mod}}

\def\EEE{\mathcal{E}}

\def\sk{\vspace{1mm}}

\makeatletter
\let\@fnsymbol\@alph
\makeatother

\title{%
Dirac operators on noncommutative hypersurfaces
}

\author{%
Hans Nguyen$^{a}$\ and\ Alexander Schenkel$^{b}$\vspace{4mm}\\
{\small School of Mathematical Sciences, University of Nottingham,}\\
{\small University Park, Nottingham NG7 2RD, United Kingdom.}\vspace{4mm}\\
{\small Email: ${}^a$~\texttt{hans.nguyen@nottingham.ac.uk}~,~~${}^b$~\texttt{alexander.schenkel@nottingham.ac.uk}}
}

\date{August 2020}

\begin{document}

\maketitle

\vspace{-5mm}

\begin{abstract}
\noindent This paper studies geometric structures on noncommutative hypersurfaces within a module-theoretic approach to noncommutative Riemannian (spin) geometry. A construction to induce differential, Riemannian and spinorial structures from a noncommutative embedding space to a noncommutative hypersurface is developed and applied to obtain noncommutative hypersurface Dirac operators. The general construction is illustrated by studying the sequence $\mathbb{T}^{2}_{\theta} \hookrightarrow \mathbb{S}^{3}_{\theta} \hookrightarrow \mathbb{R}^{4}_{\theta}$ of noncommutative hypersurface embeddings.
\end{abstract}

\vspace{-1mm}

\vspace{-2mm}

\vspace{-1mm}

\tableofcontents



\section{\label{sec:intro}Introduction and summary}
Dirac operators play a fundamental role in both 
quantum physics and noncommutative geometry. 
From the point of view of Connes' axiomatization of 
noncommutative Riemannian spin manifolds in terms 
of spectral triples \cite{Connes}, a Dirac operator 
is the basic object that is supposed to encode all
geometric information about the noncommutative space. 
However, the way in which a Dirac operator encodes
this geometric data is rather implicit, hence it is 
in general difficult to extract information about the
metric or curvature of a noncommutative space, 
see e.g.\ \cite{CM}. An alternative approach to 
noncommutative Riemannian (spin) geometry 
is to encode the relevant geometric data
layer by layer in terms of noncommutative generalizations
of differential calculi, metrics, connections and spinorial 
structures, see e.g.\ \cite{Landi}, \cite{DV} and \cite{BMbook}.
This module-theoretic approach maintains closer ties to the structures 
familiar from classical differential geometry, which can be very beneficial
for constructing, analyzing and also interpreting examples of 
noncommutative spaces. Moreover, due to results by Beggs and Majid 
\cite{BM}, this approach leads under certain additional hypotheses to 
examples of spectral triples in the sense of \cite{Connes}. 
\sk

The aim of the present paper is to develop techniques that allow us
to induce differential, Riemannian and spinorial structures from a noncommutative
embedding space to a noncommutative hypersurface. Our construction
is a noncommutative generalization of well-known results in classical differential
geometry, see e.g.\ \cite{Bures,Trautman,Bar,HMZ}, and it results in
an explicit formula for the Dirac operator on the noncommutative hypersurface.
In particular, our techniques and results can be applied to construct examples of curved 
noncommutative hypersurfaces and their Dirac operators from very simple flat 
noncommutative embedding spaces. 
\sk

The outline of the remainder of this paper is as follows:
In Section \ref{sec:prelim} we provide a brief review of the relevant
algebraic and geometric preliminaries from the module-theoretic approach
to noncommutative Riemannian (spin) geometry. Section \ref{sec:construction}
presents our main results on induced differential, Riemannian and spinorial structures
on noncommutative hypersurfaces in the sense of Definition \ref{def:hypersurface}. Our construction requires certain additional
hypotheses on the structure of the noncommutative hypersurface under consideration,
which we will introduce consecutively as soon as they are needed. We refer the reader to
Assumptions  \ref{assu:dftransparent}, \ref{assu:Pitransparent} 
and \ref{assu:nabladdfcentral} for a complete list of these hypotheses.
The main result of this section is Proposition \ref{prop:induceddiracop},
where we derive an explicit expression for the Dirac operator on the noncommutative
hypersurface. In Section \ref{sec:examples} we illustrate our constructions and results
by applying them to the sequence of noncommutative hypersurface embeddings 
$\bbT^{2}_{\theta} \hookrightarrow \bbS^{3}_{\theta} \hookrightarrow \bbR^{4}_{\theta}$
studied by Arnlind and Norkvist \cite{AN}. Starting from a very simple flat 
noncommutative geometry on $\bbR^4_\theta$, we compute the induced geometric
structures on both the Connes-Landi sphere $\bbS^3_\theta$ and the noncommutative
torus $\bbT^2_\theta$. Our induced hypersurface Dirac operators on $\bbS^3_\theta$ and 
$\bbT^2_\theta$ are isospectral to the commutative ones and they coincide with the 
Dirac operators obtained from toric deformations in \cite{CL,CDV,Brain}.


\section{\label{sec:prelim}Algebraic and geometric preliminaries}
In this paper all vector spaces, algebras, modules, etc., will be over the
field $\bbC$ of complex numbers. Given an (associative and unital) algebra
$A$, we denote by ${}_A\Mod$ the category of left $A$-modules
and by ${}_A\Mod_A$ the category of $A$-bimodules. Recall that the latter
category is monoidal with respect to the relative tensor product 
$V\otimes_A W\in{}_A\Mod_A$ of $A$-bimodules $V,W\in {}_A\Mod_A$
and monoidal unit given by the $1$-dimensional free $A$-bimodule $A\in{}_A\Mod_A$.
Furthermore, ${}_A\Mod$ is a (left) module category over the monoidal category
$({}_A\Mod_A,\otimes_A,A)$, with left action given by the relative tensor product
$V\otimes_A \EEE\in{}_A\Mod$, for all $V\in{}_A\Mod_A$ and $\EEE\in{}_A\Mod$.
\sk

Let us recall briefly some basic concepts from noncommutative geometry, see 
e.g.\ \cite{Landi}, \cite{DV} and \cite{BMbook} for a detailed introduction to the relevant frameworks.
\begin{defi}\label{def:differential}
A {\em (first-order) differential calculus} on an algebra $A$ is
a pair $(\Omega^1_A,\dd)$ consisting of an $A$-bimodule $\Omega^1_A \in {}_A\Mod_A$
and a linear map $\dd : A\to \Omega^1_A$ (called {\em differential}), such that
\begin{itemize}
\item[(i)] $\dd(a\,a^\prime) = (\dd a) \,a^\prime + a\, (\dd a^\prime)$, for all $a,a^\prime\in A$,
\item[(ii)] $\Omega^1_A = A\,\dd(A) := \big\{\sum_i a_i\, \dd a_i^\prime \,:\, a_i,a_i^\prime \in A\big\}$.
\end{itemize}
We call $\Omega^1_A$ the $A$-bimodule of {\em $1$-forms} on $A$.
\end{defi}

\begin{defi}\label{def:Connections}
Let $(\Omega^1_A,\dd)$  be a differential calculus on an algebra $A$. 
\begin{itemize}
\item[(i)] A {\em connection} on a left $A$-module $\EEE\in{}_A\Mod$ is a linear map 
$\nabla : \EEE\to \Omega^1_A\otimes_A\EEE$ that satisfies the left Leibniz rule
\begin{flalign}\label{eqn:leftLeibniz}
\nabla(a\,s) \, =\, a\,\nabla(s) + \dd a\otimes_A s\quad,
\end{flalign}
for all $a\in A$ and $s\in \EEE$. 

\item[(ii)] A {\em bimodule connection} on an $A$-bimodule
$V\in {}_A\Mod_{A}$ is a pair $(\nabla,\sigma)$ consisting of a
linear map $\nabla : V \to \Omega^1_A\otimes_A V$ and an $A$-bimodule
isomorphism $\sigma : V\otimes_A\Omega^1_A \to \Omega^1_A\otimes_A V$,
such that the following left and right Leibniz rules
\begin{subequations}
\begin{flalign}
\nabla(a\,v) &\,=\, a\,\nabla(v) + \dd a\otimes_A v\quad,\\
\nabla(v\,a)& \,=\, \nabla(v)\,a + \sigma(v\otimes_A \dd a )\quad,
\end{flalign}
\end{subequations}
are satisfied, for all $a\in A$ and $v\in V$.
\end{itemize}
\end{defi}

The concept of bimodule connections is motivated by the following
standard result, see e.g.\ \cite[Section 10]{DV}.
\begin{propo}\label{prop:tensorconnection}
Let $(\Omega^1_A,\dd)$ be a differential calculus on an algebra $A$. 
\begin{itemize}
\item[(i)] Let $\nabla^\EEE$ be a connection on a left $A$-module $\EEE\in{}_A\Mod$
and $(\nabla^V,\sigma^V)$ a bimodule connection on an $A$-bimodule
$V\in{}_A\Mod_A$. Then
\begin{flalign}
\nabla^\otimes(v\otimes_A s) \,:=\, \nabla^V(v)\otimes_A s + (\sigma^V\otimes_A \id)\big(v\otimes_A \nabla^\EEE(s)\big)\quad,
\end{flalign}
for all $v\in V$ and $s\in \EEE$, defines a connection on the tensor product module $V\otimes_A\EEE\in{}_A\Mod$.

\item[(ii)] Let $(\nabla^V,\sigma^V)$ and $(\nabla^W,\sigma^W)$ be bimodule connections on two 
$A$-bimodules $V,W\in{}_A\Mod_A$.
Then
\begin{subequations}
\begin{flalign}
\nabla^\otimes(v\otimes_A w) \,:=\, \nabla^V(v)\otimes_A w + (\sigma^V\otimes_A \id)\big(v\otimes_A \nabla^W(w)\big)\quad,
\end{flalign}
for all $v\in V$ and $w\in W$, and the composite $A$-bimodule isomorphism
\begin{flalign}
\sigma^\otimes \,:\, \xymatrix@C=3.5em{
V\otimes_AW\otimes_A\Omega^1_A \ar[r]^-{\id\otimes_A \sigma^W}~&~V\otimes_A\Omega^1_A\otimes_A W
\ar[r]^-{\sigma^V\otimes_A\id} ~&~\Omega^1_A\otimes_AV\otimes_A W
}
\end{flalign}
\end{subequations}
defines a bimodule connection on the tensor product bimodule
$V\otimes_A W\in{}_A\Mod_A$.
\end{itemize}
\end{propo}

\begin{defi}\label{def:metric}
A (generalized) {\em metric} on $\Omega^1_A$ is an $A$-bimodule map
$g :A\to \Omega^1_A \otimes_A \Omega^1_A$ for which there 
exists an $A$-bimodule map $g^{-1} : \Omega^1_A\otimes_A\Omega^1_A \to A$,
such that the two compositions
\begin{flalign}\label{eqn:metricinverse}
\xymatrix@C=3.5em@R=0.5em{
\Omega^1_A\,\cong\, \Omega^1_A \otimes_A A \ar[r]^-{\id\otimes_Ag}~&~ 
\Omega^1_A\otimes_A\Omega^1_A\otimes_A \Omega^1_A \ar[r]^-{g^{-1}\otimes_A\id}~
&~A\otimes_A\Omega^1_A \,\cong\, \Omega^1_A\\
\Omega^1_A\,\cong\, A\otimes_A \Omega^1_A \ar[r]^-{g\otimes_A \id}~&~ 
\Omega^1_A\otimes_A\Omega^1_A\otimes_A \Omega^1_A \ar[r]^-{\id\otimes_A g^{-1}}~
&~\Omega^1_A \otimes_A A\,\cong\, \Omega^1_A
}
\end{flalign}
are the identity morphisms. We call $g^{-1}$ the {\em inverse metric}.
\end{defi}

\begin{rem}\label{rem:metricindices}
Since $A$ is a free module with a basis given by the unit element $1\in A$,
the datum of an $A$-bimodule map $g :A\to \Omega^1_A \otimes_A \Omega^1_A$
is equivalent to that of a central element $g(1)\in \Omega^1_A\otimes_A\Omega^1_A$,
i.e.\ $a\, g(1) = g(1)\,a$ for all $a\in A$. Writing this element as
$g(1) = \sum_{\alpha} g^\alpha\otimes_A g_\alpha $,
the two conditions in \eqref{eqn:metricinverse} read as
\begin{flalign}
\sum_{\alpha} g^{-1} (\omega\otimes_A g^\alpha)\,g_\alpha = \omega = \sum_\alpha g^\alpha\,g^{-1}(g_\alpha\otimes_A \omega)\quad,
\end{flalign}
for all $\omega\in\Omega^1_A$. Using these identities it is easy to prove that,
provided it exists, the inverse metric $g^{-1}$ is unique.
\end{rem}

\begin{defi}\label{def:Riemannian}
Let $(\Omega^1_A,\dd)$ be a differential calculus on an algebra $A$. 
A {\em Riemannian structure} on $(\Omega^1_A,\dd)$ is a pair
$(g,(\nabla,\sigma))$ consisting of a (generalized) metric $g$ on $\Omega^1_A$
and a bimodule connection $(\nabla,\sigma)$ on $\Omega^{1}_{A}$ that satisfies
the following properties:
\begin{itemize}
\item[(i)] {\em Symmetry:}  The diagram
\begin{flalign}\label{eqn:symmetry}
\xymatrix@C=2em{
\ar[dr]_-{g^{-1}}\Omega^{1}_{A} \otimes_A \Omega^{1}_{A} \ar[rr]^-{\sigma} ~&~ ~&~\Omega^{1}_{A}\otimes_A\Omega^{1}_{A} \ar[dl]^-{g^{-1}}\\
~&~ A ~&~
}
\end{flalign}
commutes.

\item[(ii)] {\em Metric compatibility:} The diagram
\begin{flalign}\label{eqn:metriccompatibility}
\xymatrix@C=4em{
\ar[d]_-{g^{-1}}\Omega^{1}_{A}\otimes_A \Omega^{1}_{A} \ar[r]^-{\nabla^\otimes} ~&~\Omega^1_A\otimes_A\Omega^{1}_{A}\otimes_A\Omega^{1}_{A}\ar[d]^-{\id\otimes_A g^{-1}}\\
A \ar[r]_-{\dd}~&~\Omega^1_A\,\cong\,\Omega^1_A\otimes_A A
}
\end{flalign}
commutes, where $\nabla^\otimes$ is the tensor product connection from Proposition \ref{prop:tensorconnection}.
\end{itemize}
\end{defi}

\begin{rem}
Note that our definition of Riemannian structures
does {\em not} include a torsion-free condition for the connection $\nabla$.
In those cases where one has a second-order differential calculus $\Omega^2_A$,
one can supplement Definition \ref{def:Riemannian}
with the torsion-free condition $T =0$, where
$T:= \wedge \circ \nabla-\dd : \Omega^1_A\to \Omega^2_A$
is the torsion tensor, see \cite{BMbook}. The reason why we
do not consider the torsion-free condition is that our 
constructions in this paper apply to connections with torsion too,
hence this condition is not needed.
\end{rem}

Next, we introduce a suitable concept of spinorial structure 
based on the module-theoretic approach by Beggs and Majid \cite{BM,BMbook}.
Let $(\Omega^1_A,\dd)$ be a differential calculus on an algebra $A$
and $(g,(\nabla,\sigma))$ a Riemannian structure on $(\Omega^1_A,\dd)$.
Consider a left $A$-module $\EEE\in{}_A\Mod$, which we interpret as the 
module of sections of a spinor bundle. This module should come endowed
with a connection $\nabla^{\sp} : \EEE\to \Omega^1_A\otimes_A \EEE$, which we
interpret as spin connection, and an $A$-module map 
$\gamma : \Omega^{1}_{A}\otimes_A\EEE\to \EEE$, which we interpret as Clifford multiplication.
These data will be required to be compatible (in the sense defined below) 
with the Riemannian structure $(g,(\nabla,\sigma))$. For later use,
let us introduce the notation
\begin{flalign}\label{eqn:gamma2}
\gamma_{[2]} \,:\, \xymatrix@C=3.5em{
\Omega^{1}_{A}\otimes_A\Omega^{1}_{A}\otimes_A\EEE \ar[r]^-{\id\otimes_A\gamma} ~&~\Omega^{1}_{A}\otimes_A\EEE\ar[r]^-{\gamma}~&~\EEE
}
\end{flalign}
for the $A$-module map obtained by iterated application of $\gamma$. Analogously,
one can define $\gamma_{[n]} : {\Omega^{1}_{A}}^{\otimes_A n}\otimes_A\EEE\to \EEE$, for all
$n\in\bbN$.
\begin{defi}\label{def:spinorial}
Let $(\Omega^1_A,\dd)$ be a differential calculus on an algebra $A$
and $(g,(\nabla,\sigma))$ a Riemannian structure on $(\Omega^1_A,\dd)$.
A {\em spinorial structure} on $(g,(\nabla,\sigma))$ is a triple $(\EEE,\nabla^{\sp},\gamma)$
consisting of a left $A$-module $\EEE\in{}_A\Mod$, a connection $\nabla^{\sp}$ on $\EEE$
and an $A$-module map $\gamma : \Omega^{1}_{A}\otimes_A\EEE\to \EEE$ that satisfies
the following properties:
\begin{itemize}
\item[(i)] {\em Clifford relations:} The diagram
\begin{flalign}\label{eqn:cliffordrels}
\xymatrix@C=4em{
\Omega^{1}_{A}\otimes_A\Omega^{1}_{A}\otimes_A\EEE\ar[d]_-{\gamma_{[2]} + \gamma_{[2]} \circ (\sigma\otimes_A\id)} \ar[r]^-{-2 g^{-1}\otimes_A\id}~&~A\otimes_A\EEE\\
\EEE\ar[ru]_-{\cong}~&~
}
\end{flalign}
commutes.

\item[(ii)] {\em Clifford compatibility:} The diagram
\begin{flalign}\label{eqn:cliffordcomp}
\xymatrix@C=4em{
\ar[d]_-{\gamma}\Omega^{1}_{A}\otimes_A \EEE \ar[r]^-{\nabla^\otimes} ~&~\Omega^1_A\otimes_A\Omega^{1}_{A}\otimes_A\EEE\ar[d]^-{\id\otimes_A\gamma}\\
\EEE\ar[r]_-{\nabla^{\sp}}~&~\Omega^1_{A}\otimes_A\EEE
}
\end{flalign}
commutes, where $\nabla^\otimes$ is the tensor product connection from Proposition \ref{prop:tensorconnection}.
\end{itemize}
We shall call the composite
\begin{flalign}\label{eqn:diracop}
D\,:\, \xymatrix@C=3.5em{
\EEE\ar[r]^-{\nabla^{\sp}}~&~\Omega^1_A\otimes_A\EEE \ar[r]^-{\gamma}~&~\EEE
}
\end{flalign}
the {\em Dirac operator} associated with the given spinorial structure.
\end{defi}

For later use, we record the following property of the Dirac operator.
\begin{propo}\label{prop:Diracproperty}
The Dirac operator \eqref{eqn:diracop} satisfies
\begin{flalign}
D(a\,s) = a\,D(s) + \gamma(\dd a\otimes_A s) \quad,
\end{flalign}
for all $a\in A$ and $s\in\EEE$.
\end{propo}
\begin{proof}
This is a direct consequence of the Leibniz rule \eqref{eqn:leftLeibniz} for $\nabla^{\sp}$
and the fact that $\gamma$ is left $A$-linear.
\end{proof}

\begin{rem}
We would like to emphasize that our definition of spinorial structures is
less general than the one by Beggs and Majid \cite{BM,BMbook},
which does {\em not} assume Clifford compatibility \eqref{eqn:cliffordcomp}.
We decided to include this additional axiom in our definition, because it is an important
guiding principle for our construction of Dirac operators on noncommutative hypersurfaces
in Section \ref{sec:construction} and it is satisfied in our examples of interest in Section \ref{sec:examples}.
\end{rem}


\section{\label{sec:construction}Induced structures on noncommutative hypersurfaces}
Throughout the whole section, we fix an algebra $A$,
a differential calculus $(\Omega^1_A,\dd)$ on
$A$ (see Definition \ref{def:differential}), a Riemannian
structure $(g,(\nabla,\sigma))$ on $(\Omega^1_A,\dd)$ (see Definition \ref{def:Riemannian})
and a spinorial structure $(\EEE,\nabla^{\sp},\gamma)$ on
$(g,(\nabla,\sigma))$ (see Definition \ref{def:spinorial}).
We interpret $A$ as (the algebra of functions on) a noncommutative embedding space,
which is endowed with a differential, Riemannian and spinorial structure.

\subsection{\label{subsec:hypersurfaces}Noncommutative hypersurfaces}
In this subsection we introduce a suitable class of noncommutative hypersurfaces
that will form the basis for our studies. Given a $2$-sided ideal $I\subset A$, 
consider the quotient algebra
\begin{flalign}\label{eqn:Balgebra}
B\,:=\, A\big/I\quad.
\end{flalign}
Associated with the quotient algebra map $q : A\to B$ is a change of base functor 
$q_! : {}_A\Mod_A \to {}_B\Mod_B$ for bimodules, which is  
given by $q_!(V)= B\otimes_A V\otimes_A B \in {}_B\Mod_B$,
for all $V\in{}_A\Mod_A$. Because $B=A/I$ is a quotient algebra 
and $q : A\to B$ is the corresponding quotient map,
there exists a natural $B$-bimodule isomorphism
\begin{flalign}\label{eqn:q!V}
q_!(V)~ \stackrel{\cong}{\longrightarrow}~ \frac{V}{IV \cup VI}~,~~[a]\otimes_A v\otimes_A [a^\prime] ~\longmapsto~[a\,v\,a^\prime]\quad,
\end{flalign}
where $IV := \{a\,v \,: \, a\in I\text{ and }v\in V\}\subseteq V$ and $VI := \{v\,a \,:\, a\in I \text{ and }v\in V\}\subseteq V$. 
Here and in the following, we use square brackets to denote equivalence classes in quotient spaces.
Applying the change of base functor on the $A$-bimodule of $1$-forms $\Omega^1_A\in {}_A\Mod_A$
is however not sufficient to define a differential calculus on $B$,
because the differential $\dd : A \to q_!(\Omega^1_A)$
does not in general  descend to the quotient $B = A/I$.
Following e.g.\ \cite[Exercise E1.4]{BMbook}, this problem is solved by 
introducing the quotient $B$-bimodule
\begin{flalign}\label{eqn:OmegaB}
\Omega^1_B \,:=\, \frac{q_!(\Omega^1_A)}{B[\dd I] B}\,\in\,{}_B\Mod_B\quad,
\end{flalign}
where $B[\dd I] B : =\{\sum_i b_i\,[\dd a_i]\, b_i^\prime \,:\, b_i,b_i^\prime\in B\text{ and } a_i\in I\}$
is the $B$-subbimodule generated by $[\dd I]\subseteq q_!(\Omega^1_A)$.
The differential $\dd : A\to \Omega^1_A$ then descends to a linear map
\begin{flalign}\label{eqn:differentialB}
\dd_B \,:\, B~\longrightarrow ~\Omega^1_B~,~~[a]~\longmapsto~[\dd a]
\end{flalign}
and we obtain
\begin{propo}\label{prop:diffcalcB}
$(\Omega^1_B,\dd_B)$ is a differential calculus on the quotient algebra $B=A/I$.
\end{propo}
\begin{proof}
The necessary properties of Definition \ref{def:differential} are inherited from the differential
calculus $(\Omega^1_A,\dd)$ on $A$, see e.g.\  \cite[Exercise E1.4]{BMbook}.
\end{proof}

The scenario introduced above is too general to interpret $q : A\to B$ as (the dual of)
an embedding of a noncommutative hypersurface $B$ into the noncommutative
embedding space $A$. In particular, it does not capture that $B$ should be 
of  ``codimension $1$'' and that we would like the existence of a  ``normalized normal vector field''
for $B$. In order to introduce an appropriate noncommutative generalization
of these concepts\footnote{We would like to thank Branimir {\'C}a{\'c}i{\'c}
for suggesting Definition \ref{def:hypersurface} to us. This allowed us to generalize
our results for noncommutative level set hypersurfaces (cf.\ Example \ref{ex:levelset}) 
from the first version of this paper.}, we note that the 
canonical quotient map $q_!(\Omega^1_A) \twoheadrightarrow \Omega^1_B$
(cf.\ \eqref{eqn:OmegaB}) gives rise to a short exact sequence of $B$-bimodules
\begin{flalign}\label{eqn:sequence}
\xymatrix{
0 \ar[r] & N^1_B:= B[\dd I]B   \ar[r] & q_!(\Omega^1_A) \ar[r] &  \Omega^1_B \ar[r] & 0
}\quad,
\end{flalign}
where $N^1_B\in{}_B\Mod_B$ is a noncommutative analogue of the conormal bundle.
\begin{defi}\label{def:hypersurface}
We say that $B = A/I$  is a (metrically co-orientable) {\em noncommutative hypersurface}
if  the $B$-bimodule $N^1_B\in {}_B\Mod_B$ admits  a $1$-dimensional 
basis $[\nu]\in N^1_B$ with $\nu \in\Omega^1_A$ a central $1$-form, i.e.\ $a\,\nu =\nu\,a$ for all $a\in A$,
that satisfies the normalization condition
\begin{flalign}\label{eqn:etanormalized}
\big[g^{-1}(\nu\otimes_A\nu)\big] \,=\, 1\,\in\,B\quad.
\end{flalign}
\end{defi}

\begin{rem}
Definition \ref{def:hypersurface} captures both the property
of $B$ being of ``codimension $1$'', which is encoded by the statement that
$N^1_B$ has  rank $1$, and the existence of a ``normalized normal vector field'',
which in our dual language of forms is given by the normalized $\nu\in \Omega^1_A$.
Throughout this paper, we shall always use the simpler terminology
of noncommutative hypersurfaces instead of the technically more appropriate, but cumbersome,
term {\em metrically co-orientable} noncommutative hypersurfaces that emphasises
existence of the normalized $1$-form $\nu$.
\end{rem}

\begin{ex}\label{ex:levelset}
An important and interesting class of examples of noncommutative hypersurfaces
in the sense of Definition \ref{def:hypersurface} is given by {\em noncommutative
level set hypersurfaces}. These are determined by $2$-sided 
ideals $I = (f)\subset A$ generated by a central element 
$f\in\mathcal{Z}(A)\subseteq A$ such that $\nu := \dd f \in \Omega^1_A$ 
is central and satisfies the normalization
condition \eqref{eqn:etanormalized}. It is easy to see that in this case
$[\nu]=[\dd f]$ defines a basis of $N^1_B =B[\dd I]B = B[\dd f]B = [\dd f]B = B[\dd f]$,
where in the last two steps we used that $\dd f $ is central.
All our examples in Section \ref{sec:examples} will be of this type.
\end{ex}

In what follows we let $B=A/I$ be any noncommutative hypersurface
in the sense of Definition \ref{def:hypersurface}.
As a preparation for the following subsections, we 
construct a splitting of the sequence \eqref{eqn:sequence} 
that determines a $B$-bimodule isomorphism between 
$\Omega^1_B\in{}_B\Mod_B$ and a certain $B$-subbimodule
of $q_!(\Omega^1_A)\in {}_B\Mod_B$. 
Using the inverse metric $g^{-1} : \Omega^1_A\otimes_A\Omega^1_A\to A$
and the normalized $1$-form $\nu\in \Omega^1_A$,
we define the $A$-bimodule endomorphism
\begin{subequations}\label{eqn:Pi}
\begin{flalign}
\Pi\,:\,\Omega^1_A ~\longrightarrow~\Omega^1_A ~,~~\omega~\longmapsto~
\omega - g^{-1}(\omega\otimes_A \nu) \, \nu\quad.
\end{flalign}
Note that $\Pi$ is indeed right $A$-linear because $\nu$ is by hypothesis central. 
Using the change of base functor, $\Pi$ defines a $B$-bimodule endomorphism
\begin{flalign}\label{eqn:q!Pi}
\Pi \, :\, q_!(\Omega^1_A)~\longrightarrow~q_!(\Omega^1_A)~,~~[\omega]~\longmapsto
~\Pi\big([\omega]\big) := \big[\Pi(\omega)\big]\quad.
\end{flalign}
\end{subequations}
\begin{propo}\label{prop:Piproperties}
The $B$-bimodule endomorphism $\Pi$ from \eqref{eqn:q!Pi} satisfies the following properties:
\begin{itemize}
\item[(i)] $\Pi\big([\nu]\big) = 0$.
\item[(ii)] $\Pi^2 =\Pi$, i.e.\ $\Pi$ is a projector.
\item[(iii)] The induced $B$-bimodule map $\Pi : \Omega^1_B \to q_!(\Omega^1_A)$ 
on $\Omega^1_B$ (cf.\ \eqref{eqn:OmegaB}) 
is a section of the quotient $B$-bimodule map $q_!(\Omega^1_A)\twoheadrightarrow \Omega^1_B $. 
In particular, it defines an isomorphism $\Omega^1_B \cong \Pi q_!(\Omega^1_A)$.
\end{itemize}
\end{propo}
\begin{proof}
Item (i) follows directly from the normalization condition \eqref{eqn:etanormalized}
and item (ii) is a direct consequence of (i). To prove item (iii), note that the induced map
$\Pi : \Omega^1_B \to q_!(\Omega^1_A)$  is well-defined because of (i) and the fact that
$[\nu]$ is by hypothesis a basis for $N^1_B$. It is a section of the quotient map
because the latter maps $[\nu] $ to $0$. This in particular implies that
the induced map $\Pi : \Omega^1_B \to q_!(\Omega^1_A)$ is injective, hence it defines
an isomorphism onto its image $\Pi  q_!(\Omega^1_A)$.
\end{proof}

\subsection{\label{subsec:riem}Induced Riemannian structure}
The aim of this subsection is to induce a Riemannian structure $(g_B,(\nabla_B,\sigma_B))$
on the differential calculus $(\Omega^1_B,\dd)$ from Section \ref{subsec:hypersurfaces}. 
Using the change of base functor, the metric $g : A\to \Omega^1_A\otimes_A\Omega^1_A$ 
on $\Omega^1_A$ defines a $B$-bimodule map
\begin{subequations}
\begin{flalign}
g \,: \, B ~\longrightarrow~q_!(\Omega^1_A)\otimes_B q_!(\Omega^1_A)~,~~
[a]~\longmapsto~\big[g(a)\big]
\end{flalign}
and the inverse metric $g^{-1} : \Omega^1_A\otimes_A \Omega^1_A\to A$ defines
a $B$-bimodule map
\begin{flalign}\label{eqn:inversemetricBtmp}
g^{-1} \,:\, q_!(\Omega^1_A)\otimes_B q_!(\Omega^1_A) ~\longrightarrow~B~,~~
[\omega]\otimes_B[\zeta] ~\longmapsto~\big[g^{-1}(\omega\otimes_A \zeta)\big]\quad.
\end{flalign}
\end{subequations}
Using also the quotient map
$q_!(\Omega^1_A) \twoheadrightarrow \Omega^1_B$ and its section $\Pi : \Omega^1_B\to q_!(\Omega^1_A) $
from Proposition \ref{prop:Piproperties} (see also \eqref{eqn:Pi}), we define
the composite $B$-bimodule maps
\begin{subequations}\label{eqn:metricandinversemetricB}
\begin{flalign}\label{eqn:metricB}
g_B \,:\ \xymatrix@C=2.5em{
B \ar[r]^-{g}~&~ q_!(\Omega^1_A)\otimes_B q_!(\Omega^1_A) \ar@{>>}[r]~
&~\Omega^1_B\otimes_B \Omega^1_B
}
\end{flalign}
and
\begin{flalign}\label{eqn:inversemetricB}
g_B^{-1} \,:\, \xymatrix@C=2.5em{
\Omega^1_B\otimes_B \Omega^1_B \ar[r]^-{\Pi\otimes_B\Pi}~&~
q_!(\Omega^1_A)\otimes_B q_!(\Omega^1_A) \ar[r]^-{g^{-1}}~
&~B
}\quad.
\end{flalign}
\end{subequations}
For our studies below, we shall need the following 
\begin{assu}\label{assu:dftransparent}
The $A$-bimodule isomorphism $\sigma : 
\Omega^1_A\otimes_A\Omega^1_A\to \Omega^1_A\otimes_A\Omega^1_A$
associated with the bimodule connection $(\nabla,\sigma)$ on $\Omega^1_A$ satisfies
\begin{flalign}\label{eqn:dftransparent}
\sigma(\omega\otimes_A \nu ) = \nu \otimes_A \omega \quad,\quad
\sigma(\nu\otimes_A \omega) =\omega\otimes_A \nu\quad,
\end{flalign}
for all $\omega\in \Omega^1_A$.
\end{assu}
\begin{lem}\label{lem:dftransparent}
Assumption \ref{assu:dftransparent} implies the following properties:
\begin{itemize}
\item[(i)] $g^{-1}(\omega \otimes_A \nu)= g^{-1}(\nu \otimes_A \omega)$, for all $\omega\in\Omega^1_A$.

\item[(ii)] $\big[g^{-1}\big(\Pi(\omega)\otimes_A \nu \big)\big] = 0 $ and 
$\big[g^{-1}\big(\nu \otimes_A\Pi(\omega)\big)\big] =0$ in $B=A/I$, for all $\omega\in\Omega^1_A$. 
This implies that
\begin{flalign}
\big[g^{-1}\big(\Pi(\omega)\otimes_A\Pi(\zeta)\big)\big] = 
\big[g^{-1}\big(\omega\otimes_A\Pi(\zeta)\big)\big] = 
\big[g^{-1}\big(\Pi(\omega)\otimes_A\zeta\big)\big]\quad,
\end{flalign}
for all $\omega,\zeta\in\Omega^1_A$.

\item[(iii)] $\big[(\id\otimes_A g^{-1})\big(\nabla(\nu)\otimes_A \nu\big)\big] = 0$.
 
\item[(iv)] The two $B$-bimodule maps in \eqref{eqn:metricandinversemetricB} define 
a metric $g_B$ and its inverse $g_B^{-1}$ on $\Omega^1_B$.
\end{itemize}
\end{lem}
\begin{proof}
Item (i) is a direct consequence of the symmetry property of $g^{-1}$ (cf.\ Definition \ref{def:Riemannian})
and Assumption \ref{assu:dftransparent}. The first equality of item (ii) follows from a short calculation
\begin{flalign}
\nn \big[g^{-1}\big(\Pi(\omega)\otimes_A \nu\big) \big]&= 
\big[g^{-1}\big(\big(\omega - g^{-1}(\omega\otimes_A\nu) \,\nu\big)\otimes_A \nu \big)  \big]\\
&=\big[g^{-1}(\omega\otimes_A\nu) - g^{-1}(\omega\otimes_A\nu)  ~ g^{-1}(\nu\otimes_A\nu)\big] =0\quad,
\end{flalign}
where we used the normalization condition \eqref{eqn:etanormalized} for $\nu$.
The second equality in item (ii) follows from this and (i). 
\sk

Item (iii) follows from the calculation
\begin{flalign}
\nn \big[(\id\otimes_A g^{-1})\big(\nabla(\nu)\otimes_A \nu\big)\big] 
&=\big[ \dd \big(g^{-1}(\nu\otimes_A\nu) \big)- (\id\otimes_A g^{-1})(\sigma\otimes_A \id) \big(\nu \otimes_A \nabla(\nu)\big) \big]\\
&= -\big[(\id\otimes_A g^{-1})\big(\nabla(\nu) \otimes_A\nu\big) \big]\quad,
\end{flalign}
where in the first step we used metric compatibility \eqref{eqn:metriccompatibility}
and in the second step we used the normalization condition \eqref{eqn:etanormalized} and item (i).
\sk

To prove item (iv), 
we use the same notations as in Remark \ref{rem:metricindices} to write
$g_B(1) = [g(1)] =  [ \sum_{\alpha} g^\alpha\otimes_A g_\alpha] =  \sum_{\alpha} [g^\alpha]\otimes_B[g_\alpha]$
and $g_B^{-1}([\omega]\otimes_B [\zeta]) = \big[g^{-1}\big(\Pi(\omega)\otimes_A \Pi(\zeta) \big)\big]$.
We then compute
\begin{flalign}
 \sum_{\alpha} g_B^{-1}\big([\omega]\otimes_B [g^\alpha]\big)\, [g_\alpha] 
= \Big[ \sum_{\alpha} g^{-1}\big(\Pi(\omega) \otimes_A g^\alpha\big)\, g_\alpha\Big] = \big[\Pi(\omega)\big] = [\omega]\quad,
\end{flalign}
where in the first step we used (ii).
The second step follows from $g^{-1}$ being the inverse metric of $g$ and the last step uses
that $\Pi$ is a section of the quotient map (cf.\ Proposition \ref{prop:Piproperties}).
The second property $ \sum_{\alpha} [g^\alpha]\, g_B^{-1}\big([g_\alpha]\otimes_B[\omega]\big) = [\omega]$
follows from a similar calculation.
\end{proof}

Let us now focus on the bimodule connection $(\nabla,\sigma)$
on $\Omega^1_A$. From  \eqref{eqn:OmegaB},
we observe that the connection $\nabla : \Omega^1_A\to\Omega^1_A\otimes_A \Omega^1_A$
descends to a connection $\nabla : q_!(\Omega^1_A)\to \Omega^1_B\otimes_B q_!(\Omega^1_A)$
on $q_!(\Omega^1_A)\in{}_B\Mod_B$. Indeed, from the left Leibniz 
rule we conclude that $[\nabla(a\,\omega)] = [a\,\nabla(\omega) + \dd a\otimes_A \omega]=0$,
for all $a\in I$, hence this map is well-defined on the quotient. 
Using the quotient map $q_!(\Omega^1_A) \twoheadrightarrow \Omega^1_B$
and its section $\Pi : \Omega^1_B\to q_!(\Omega^1_A)$ from Proposition \ref{prop:Piproperties}
(see also \eqref{eqn:Pi}), we define the composite linear map
\begin{flalign}\label{eqn:nablaB}
\nabla_B\,:\, \xymatrix@C=2.5em{
\Omega^1_B \ar[r]^-{\Pi}~&~
q_!(\Omega^1_A) \ar[r]^-{\nabla}~&~\Omega^1_B\otimes_B q_!(\Omega^1_A) \ar@{->>}[r]~&~
\Omega^1_B\otimes_B\Omega^1_B
}\quad.
\end{flalign}
The $A$-bimodule isomorphism $\sigma : \Omega^1_A\otimes_A\Omega^1_A\to
\Omega^1_A\otimes_A\Omega^1_A$ associated with the bimodule connection
$(\nabla,\sigma) $ on $\Omega^1_A$ descends, due to Assumption
\ref{assu:dftransparent}, to the $B$-bimodule isomorphism
\begin{flalign}\label{eqn:sigmaB}
\sigma_B\,:\, \Omega^1_B\otimes_B\Omega^1_B~\longrightarrow~\Omega^1_B\otimes_B\Omega^1_B~,~~
[\omega]\otimes_B[\zeta]~\longmapsto~\big[\sigma(\omega\otimes_A\zeta)\big]\quad.
\end{flalign}
\begin{lem}
The pair $(\nabla_B,\sigma_B)$ defined by \eqref{eqn:nablaB} and \eqref{eqn:sigmaB}
is a bimodule connection on $\Omega^1_B$. It reads explicitly as
\begin{flalign}\label{eqn:nablaBexplicit}
\nabla_B\big([\omega]\big) \,=\, \big[\nabla(\omega)  - g^{-1}(\omega\otimes_A\nu)\,\nabla(\nu) \big]\quad,
\end{flalign}
for all $[\omega]\in \Omega^1_B$.
\end{lem}
\begin{proof}
The explicit expression \eqref{eqn:nablaBexplicit} is obtained by a 
short calculation
\begin{flalign}
\nn \nabla_B\big([\omega]\big) &= \big[\nabla\big(\omega - g^{-1}(\omega\otimes_A\nu)\,\nu\big) \big]
= \big[\nabla(\omega) - g^{-1}(\omega\otimes_A\nu) \,\nabla(\nu) - \dd \big(g^{-1}(\omega\otimes_A\nu)\big)\otimes_A\nu \big]\\
&=\big[\nabla(\omega) - g^{-1}(\omega\otimes_A\nu) \,\nabla(\nu)\big]\quad,
\end{flalign}
where in the second step we used the left Leibniz rule for $\nabla$
and in the third step that $\nu $ is identified with $0$
in $\Omega^1_B$ (cf.\ \eqref{eqn:OmegaB}). The left Leibniz rule is a direct
consequence of this expression and the right Leibniz rule follows
from the fact that $\nabla(\nu)\in \Omega^1_A\otimes_A\Omega^1_A$
is a central element. The latter statement is proven as follows
\begin{flalign}\label{eqn:nabladdfcentral}
a\,\nabla(\nu) = \nabla(a\,\nu) - \dd a\otimes_A\nu
= \nabla(\nu\,a) - \sigma(\nu\otimes_A\dd a) = \nabla(\nu)\,a\quad,
\end{flalign}
where we used the left and right Leibniz rules for $(\nabla,\sigma)$,
centrality of $\nu$ and  Assumption \ref{assu:dftransparent}.
\end{proof}
\begin{rem}
Note that \eqref{eqn:nablaBexplicit} is a noncommutative analog
of the usual Gauss formula for connections on Riemannian submanifolds, 
see e.g.\ \cite[Chapter VII.3]{KN}.
\end{rem}

In order to prove the main result of this subsection, we require an additional
\begin{assu}\label{assu:Pitransparent}
The diagrams
\begin{subequations}
\begin{flalign}
\xymatrix@C=3.5em{
\ar[d]_-{\sigma}q_!(\Omega^1_A)\otimes_B q_!(\Omega^1_A) \ar[r]^-{\Pi\otimes_B\id} ~&~q_!(\Omega^1_A) \otimes_B q_!(\Omega^1_A)\ar[d]^-{\sigma}\\
\ar[r]_-{\id\otimes_B\Pi} q_!(\Omega^1_A)\otimes_B q_!(\Omega^1_A)~&~q_!(\Omega^1_A) \otimes_B q_!(\Omega^1_A)
}
\end{flalign}
and
\begin{flalign}
\xymatrix@C=3.5em{
\ar[d]_-{\sigma}q_!(\Omega^1_A)\otimes_B q_!(\Omega^1_A) \ar[r]^-{\id\otimes_B\Pi} ~&~q_!(\Omega^1_A) \otimes_B q_!(\Omega^1_A)\ar[d]^-{\sigma}\\
\ar[r]_-{\Pi\otimes_B \id} q_!(\Omega^1_A)\otimes_B q_!(\Omega^1_A)~&~q_!(\Omega^1_A) \otimes_B q_!(\Omega^1_A)
}
\end{flalign}
\end{subequations}
commute. 
\end{assu}
\begin{propo}\label{prop:RiemannianStructure}
The pair $(g_B,(\nabla_B,\sigma_B))$ defined in \eqref{eqn:metricandinversemetricB},
\eqref{eqn:nablaB} and \eqref{eqn:sigmaB} is a Riemannian structure
on $(\Omega^1_B,\dd)$.
\end{propo}
\begin{proof}
It remains to prove the symmetry and metric compatibility properties from Definition \ref{def:Riemannian}.
The symmetry property \eqref{eqn:symmetry} for $(g_B,(\nabla_B,\sigma_B))$ follows immediately
from Assumption \ref{assu:Pitransparent} and symmetry of $g^{-1}$.
To verify metric compatibility \eqref{eqn:metriccompatibility} for $(g_B,(\nabla_B,\sigma_B))$, 
we compute by using metric compatibility of the original Riemannian structure $(g,(\nabla,\sigma))$
\begin{flalign}
\dd_B \big(g^{-1}_B\big([\omega]\otimes_B[\zeta]\big)\big)
= \Big[(\id\otimes_Ag^{-1})\Big(\nabla\Pi(\omega) \otimes_A \Pi(\zeta) + \sigma_{12}\big(\Pi(\omega)\otimes_A \nabla\Pi(\zeta)\big)\Big)\Big]\quad,
\end{flalign}
where $\sigma_{12} := \sigma\otimes_A\id$. Using Lemma \ref{lem:dftransparent} (ii), we can in
the first term replace $\nabla\Pi(\omega)$ with $(\id\otimes_A\Pi)\nabla\Pi(\omega)$.
Using also Assumption \ref{assu:Pitransparent}, we can replace in the second term
$\sigma_{12}\big(\Pi(\omega)\otimes_A \nabla\Pi(\zeta)\big)$ with $ (\id\otimes_A\Pi\otimes_A\id)
\sigma_{12}\big(\omega\otimes_A \nabla\Pi(\zeta)\big)$ and hence via 
Lemma \ref{lem:dftransparent} (ii) with $ (\id\otimes_A\Pi\otimes_A\Pi)
\sigma_{12}\big(\omega\otimes_A \nabla\Pi(\zeta)\big)$.
The resulting expression proves metric compatibility for $(g_B,(\nabla_B,\sigma_B))$.
\end{proof}

\subsection{\label{subsec:spinor}Induced spinorial structure}
We now induce a spinorial structure $(\EEE_B,\nabla^\sp_B, \gamma_B)$
on the Riemannian structure $(g_B,(\nabla_B,\sigma_B))$
from Section \ref{subsec:riem}. Our definitions and constructions
below are inspired by well-known results on spinorial structures on hypersurfaces 
in classical differential geometry, see e.g.\ \cite{Bures,Trautman,Bar} and also \cite{HMZ} 
for a good review.
As the first step, we use the change of base functor (for left modules) to define
the $B$-module
\begin{flalign}\label{eqn:EEEB}
\EEE_B \,:=\, q_!(\EEE)\,\cong \, \frac{\EEE}{I \EEE} \,\in\,{}_B\Mod\quad.
\end{flalign}
To introduce a suitable Clifford multiplication $\gamma_B : \Omega^1_B \otimes_B \EEE_B\to \EEE_B$, 
we recall the classical case from \cite[Eqn.\ (3.4)]{HMZ} and define the $B$-module map
\begin{flalign}\label{eqn:gammaB}
\gamma_B\,:\, \Omega^1_B\otimes_B\EEE_B~\longrightarrow~\EEE_B~,~~
[\omega]\otimes_B [s] ~\longmapsto~\big[\gamma_{[2]}\big(\Pi(\omega)\otimes_A \nu \otimes_A s\big)\big]\quad,
\end{flalign}
where $\gamma_{[2]}$ was defined in \eqref{eqn:gamma2}.
Note that this map is well-defined since the normalized $1$-form
$\nu \in\Omega^1_A$ is central by Definition \ref{def:hypersurface}.
\sk

The given connection $\nabla^\sp : \EEE\to\Omega^1_A\otimes_A\EEE$
on $\EEE\in{}_A\Mod$ descends to a connection
$\nabla^\sp : \EEE_B\to \Omega^1_B\otimes_B\EEE_B$
on $\EEE_B\in{}_B\Mod$ because $[\nabla^\sp (a\,s) = a\,\nabla^\sp(s) + \dd a\otimes_A s] =0$,
for all $a\in I$, as a consequence of  the relations in \eqref{eqn:OmegaB} and \eqref{eqn:EEEB}.
This is however not yet the correct induced spin connection on $\EEE_B$.
Motivated by the classical spinorial Gauss formula from \cite[Eqn.\ (3.5)]{HMZ},
we define
\begin{flalign}\label{eqn:nablaspB}
\nabla^\sp_B\,:\, \EEE_B~\longrightarrow~\Omega^1_B\otimes_B\EEE_B~,~~
[s]~\longmapsto~\Big[\nabla^\sp(s) + \frac{1}{2} (\id\otimes_A\gamma_{[2]})\big(\nabla (\nu) \otimes_A \nu
\otimes_A s\big)\Big]\quad.
\end{flalign}
This defines a connection on the left $B$-module $\EEE_B\in{}_B\Mod$
since both $\nu\in\Omega^1_A$ and $\nabla(\nu)\in \Omega^1_A\otimes_A \Omega^1_A$
are central. (The latter statement was proven in \eqref{eqn:nabladdfcentral}.)
\sk

In order to prove that these data define a spinorial structure in the sense
of Definition \ref{def:spinorial}, we require an additional
\begin{assu}\label{assu:nabladdfcentral}
The element $\nabla(\nu)\in\Omega^1_A\otimes_A\Omega^1_A$ satisfies
\begin{flalign}
\big[\sigma_{23}\sigma_{12}\big(\Pi(\omega)\otimes_A \nabla(\nu)\big)\big]
\,=\, \big[\nabla(\nu)\otimes_A \Pi(\omega)\big]\,\in\,\Omega^1_B\otimes_Bq_!(\Omega^1_A)\otimes_Bq_!(\Omega^1_A)\quad,
\end{flalign}
for all $\omega\in\Omega^1_A$,  where $\sigma_{12} :=\sigma\otimes_A\id$
and $\sigma_{23} := \id\otimes_A \sigma$.
\end{assu}

\begin{propo}\label{prop:SpinStr}
The triple $(\EEE_B,\nabla^\sp_B,\gamma_B)$ defined in \eqref{eqn:EEEB},
\eqref{eqn:nablaspB} and \eqref{eqn:gammaB} is a spinorial structure 
on the Riemannian structure $(g_B,(\nabla_B,\sigma_B))$.
\end{propo}
\begin{proof}
It remains to prove the Clifford relations and Clifford compatibility
properties from Definition \ref{def:spinorial}. In these calculations we
frequently use the identities
\begin{subequations}\label{eqn:cliffordetarels}
\begin{flalign}\label{eqn:cliffordetarels1}
\big[ \gamma_{[2]}\big(\Pi(\omega)\otimes_A \nu \otimes_A s\big)\big]
= -[\gamma_{[2]}\big(\nu\otimes_A\Pi(\omega) \otimes_A s\big)]
\end{flalign}
and
\begin{flalign}\label{eqn:cliffordetarels2}
\big[ \gamma_{[2]}\big(\nu\otimes_A \nu \otimes_A s\big)\big] = -[s]\quad,
\end{flalign}
\end{subequations}
which follow from the Clifford relations \eqref{eqn:cliffordrels} for $\gamma$,
Assumption \ref{assu:dftransparent}, Lemma \ref{lem:dftransparent} (ii) 
and the normalization condition \eqref{eqn:etanormalized}.
\sk

The Clifford relations \eqref{eqn:cliffordrels} for $\gamma_B$ follow from
a direct calculation, for which we introduce the convenient short notation
$\sigma(\omega \otimes_A \zeta) = \sum_{\alpha} \zeta^\alpha\otimes_A \omega_\alpha$. We compute
\begin{flalign}
\nn &\gamma_{B[2]}\Big([\omega]\otimes_B[\zeta]\otimes_B[s] + {\sigma_{B}}_{12}\big([\omega]\otimes_B[\zeta]\otimes_B[s]\big)\Big)\\
\nn &\qquad\qquad=\Big[\gamma_{[4]}\Big(\Pi(\omega) \otimes_A \nu\otimes_A \Pi(\zeta) \otimes_A\nu\otimes_A s + 
 \sum_{\alpha} \Pi(\zeta^\alpha)\otimes_A\nu\otimes_A\Pi(\omega_\alpha)\otimes_A\nu\otimes_A s\Big)\Big]\\
\nn &\qquad\qquad=\Big[\gamma_{[2]}\Big(\Pi(\omega) \otimes_A \Pi(\zeta)\otimes_A s +  \sum_{\alpha} \Pi(\zeta^\alpha) 
\otimes_A\Pi(\omega_\alpha)\otimes_A s\Big)\Big] \\
&\qquad\qquad =
-2\,g_B^{-1}\big([\omega]\otimes_B[\zeta]\big)\,[s]\quad,
\end{flalign}
where in the second step we used \eqref{eqn:cliffordetarels}.
The last step follows from Assumption \ref{assu:Pitransparent},
the Clifford relations for $\gamma$ and the definition of $g_B^{-1}$ in \eqref{eqn:inversemetricB}.
\sk

Proving the Clifford compatibility property \eqref{eqn:cliffordcomp} for
$\nabla_B$, $\nabla^\sp_B$ and $\gamma_B$ is a lengthier computation.
Using as above \eqref{eqn:cliffordetarels}, Assumption \ref{assu:Pitransparent} 
and also Clifford compatibility for $\nabla$, $\nabla^\sp$
and $\gamma$, one finds that the desired equality
$\nabla_B^\sp\,\gamma_B\big([\omega]\otimes_B [s]\big)
=(\id\otimes_B\gamma_{B})\big(\nabla_B^\otimes\big([\omega]\otimes_B[s]\big)\big)$
is equivalent to the statement that the two expressions
\begin{subequations}
\begin{flalign}\label{eqn:TMP1}
\Big[(\id\otimes_A\gamma_{[2]})\Big(\sigma_{12}\big(\Pi(\omega)\otimes_A\nabla(\nu) \otimes_A s\big)
+ \frac{1}{2}\nabla(\nu)\otimes_A\Pi(\omega)\otimes_A s\Big)\Big]
\end{flalign}
and
\begin{flalign}\label{eqn:TMP2}
\Big[(\id\otimes_A g^{-1})\big(\nabla\Pi(\omega) \otimes_A \nu\big)\otimes_A s
+\frac{1}{2} (\id\otimes_A\gamma_{[4]})\Big(\sigma_{12}\sigma_{23}\big(\Pi(\omega) \otimes_A\nu\otimes_A \nabla(\nu)\otimes_A\nu\otimes_A s\big) \Big)\Big]
\end{flalign}
\end{subequations}
are equal. (The term with $g^{-1}$ in \eqref{eqn:TMP2}
arises from computing $(\id\otimes_A\Pi)\nabla\Pi(\omega) 
= \nabla\Pi(\omega) - ( \id\otimes_A g^{-1})\big(\nabla\Pi(\omega)\otimes_A\nu\big)\otimes_A\nu$
via \eqref{eqn:Pi}.)
Using metric compatibility \eqref{eqn:metriccompatibility} for $(g,(\nabla,\sigma))$,
Lemma \ref{lem:dftransparent} (ii) and the Clifford relations for $\gamma$, we can rewrite
the first term of \eqref{eqn:TMP2} as
\begin{flalign}
\nn \text{\eqref{eqn:TMP2}}{}^{\mathrm{1st}} &=
\Big[-(\id\otimes_A g^{-1})\sigma_{12}\big(\Pi(\omega) \otimes_A \nabla(\nu)\big)\otimes_A s\Big]\\
&=\Big[\frac{1}{2}(\id\otimes_A\gamma_{[2]})\Big(\sigma_{12}\big(\Pi(\omega)\otimes_A\nabla(\nu)\otimes_A s\big) + \sigma_{23}\sigma_{12}\big(\Pi(\omega)\otimes_A\nabla(\nu)\otimes_A s\big)\Big)\Big]\quad.
\end{flalign}
Concerning the second term of \eqref{eqn:TMP2}, we use the Clifford relations
for $\gamma$ to bring the left factor of $\nu$ to the right and observe that there is
no $g^{-1}$ contribution as a result of Lemma \ref{lem:dftransparent} (iii). Hence, we can rewrite
the second term of \eqref{eqn:TMP2} as
\begin{flalign}
\text{\eqref{eqn:TMP2}}{}^{\mathrm{2nd}} &=
\Big[\frac{1}{2}(\id\otimes_A\gamma_{[2]})\Big(\sigma_{12}\big(\Pi(\omega)\otimes_A\nabla(\nu)\otimes_A s\big) \Big)\Big]\quad.
\end{flalign}
From these simplifications and Assumption \ref{assu:nabladdfcentral},
it follows that the expressions in \eqref{eqn:TMP2} and \eqref{eqn:TMP1} are equal.
This completes our proof of the Clifford compatibility property.
\end{proof}

We conclude this section by presenting an explicit expression for the induced
Dirac operator
\begin{flalign}\label{eqn:induceddiracop}
	D_{B}\,:\, \xymatrix@C=3.5em{
\EEE_{B}\ar[r]^-{\nabla^{\sp}_{B}}~&~\Omega^1_B\otimes_B\EEE_{B} \ar[r]^-{\gamma_{B}}~&~\EEE_{B}
}\quad.
\end{flalign}
\begin{propo}\label{prop:induceddiracop}
The induced Dirac operator \eqref{eqn:induceddiracop} reads explicitly as
\begin{flalign}\label{eqn:induceddiracopexplicit}
D_B\big([s]\big)=\Big[ -\frac{1}{2}\big(\gamma_{[2]} - \gamma_{[2]}\,(\sigma\otimes_A\id)\big) \big(\nu\otimes_A \nabla^\sp(s) \big) 
+\frac{1}{2}\gamma_{[2]}\big((\Pi\otimes_A\id)\nabla(\nu)\otimes_A s\big)\Big]\quad,
\end{flalign}
for all $[s]\in\EEE_B$.
\end{propo}
\begin{proof}
This is a straightforward calculation using \eqref{eqn:nablaspB}, \eqref{eqn:gammaB},
the projector \eqref{eqn:Pi} and the Clifford relations \eqref{eqn:cliffordrels} for $\gamma$, in particular
 \eqref{eqn:cliffordetarels}. Since the relevant steps are similar to those in
the proof of Proposition \ref{prop:SpinStr}, we do not have to write out the details of this calculation.
\end{proof}


\section{\label{sec:examples}Examples}
In this section we will illustrate our framework by applying it to the 
sequence of noncommutative hypersurface embeddings 
$\bbT^{2}_{\theta} \hookrightarrow \bbS^{3}_{\theta} \hookrightarrow \bbR^{4}_{\theta}$
studied by Arnlind and Norkvist \cite{AN}. We describe first 
the relevant differential, Riemannian and spinorial structures on the noncommutative
embedding space $\bbR^{4}_{\theta}$ following our definitions in Section \ref{sec:prelim}. 
We then use our constructions from Section \ref{sec:construction}
to induce these structures to the noncommutative hypersurface
$ \bbS^{3}_{\theta} \hookrightarrow \bbR^{4}_{\theta}$
and in a second step to the noncommutative hypersurface 
$\bbT^{2}_{\theta} \hookrightarrow \bbS^{3}_{\theta}$.
These studies result in explicit expressions for the Dirac operators (in the sense of
Definition \ref{def:spinorial}) on these noncommutative hypersurfaces.

\subsection{\label{subsec:R4theta}Noncommutative embedding space $\bbR^4_\theta$}
We begin by introducing the noncommutative embedding space $\bbR^{4}_{\theta}$.
Instead of working with real coordinates $x^{\mu}$, for $\mu = 1,\dots , 4$,
it will be more convenient to introduce the complex coordinates 
$z^1:= x^1 + i x^2$ and $z^2 := x^3 + i x^4$,
together with their complex conjugates
$z^3 := \overline{z^1}= x^1 - i x^2$ and $z^4 := \overline{z^2}=x^3-i x^4$.
The noncommutative embedding space $\bbR^{4}_{\theta}\cong \bbC^2_\theta$ is then
described by the noncommutative algebra
\begin{flalign}\label{eqn:R4algebra}
A \, :=\,  \frac{\bbC[z^{1},z^{2},z^{3},z^{4}]}{(z^{i}\, z^{j} - R^{ji}\, z^{j}\, z^{i})}
\end{flalign}
that is freely generated by the complex coordinates, modulo the ideal generated by
the commutation relations determined by the entries $R^{ji}$ of the matrix
\begin{flalign}\label{eqn:Rmatrix}
R \,:=\, 
\begin{pmatrix}
	1 & e^{-i\theta} & 1 & e^{i\theta} \\
	e^{i\theta} & 1 & e^{-i\theta} & 1 \\
	1 & e^{i\theta} & 1 & e^{-i\theta} \\
	e^{-i\theta} & 1 & e^{i\theta} & 1
\end{pmatrix}\quad,\quad \theta\in\bbR\quad.
\end{flalign}
For later use, we note that the entries of the matrix $R$
satisfy
\begin{subequations}
\begin{flalign}\label{eqn:Rhermitian}
R^{ij}  \,=\, \overline{R^{ji}}
\end{flalign}
and
\begin{flalign}\label{eqn:RijRji}
R^{ij}\,R^{ji} \,=\, 1\quad,
\end{flalign} 
\end{subequations}
where in the latter equation there is no summation over $i$ and $j$.
\sk

To define a differential calculus on $A$, let us introduce the free {\em left} $A$-module 
\begin{subequations}\label{eqn:Rzdz}
\begin{flalign}
\Omega^{1}_{A} \,:= \, \bigoplus_{i = 1}^{4} A\, \dd z^{i}\quad,
\end{flalign}
which we endow with the right $A$-action determined by
\begin{flalign}
\dd z^{i} \, z^{j} \,:=\, R^{ji}\, z^{j}\, \dd z^{i}\quad.
\end{flalign}
\end{subequations}
(Note that this is analogous to the commutation relations in \eqref{eqn:R4algebra}.)
This defines an $A$-bimodule $\Omega^1_A\in{}_A\Mod_A$,
which we endow with a differential $\dd : A \to \Omega^{1}_{A}$ 
by setting  $\dd : z^{i} \mapsto \dd z^{i}$ for the generators 
and extending to all of $A$ via the Leibniz rule.
\begin{propo}\label{prop:R4diffcalc}
The pair $(\Omega^{1}_{A},\dd)$ is a differential calculus on $A$.
\end{propo}
\begin{proof}
The statement holds by construction.
\end{proof}

The next step is to introduce a Riemannian structure.
For this we consider the standard (flat) Euclidean metric on $\bbR^4_\theta$,
which in complex coordinates reads as
\begin{subequations}\label{eqn:R4metric}
\begin{flalign}
g \,:=\,  \sum_{i,j = 1}^{4}g_{ij} \, \dd z^{i}\otimes_{A}\dd z^{j}\, \in \,\Omega^1_A\otimes_A\Omega^1_A\quad,
\end{flalign}
where $g_{ij}$ are the entries of the matrix
\begin{flalign}
(g_{ij}) \, :=\, \frac{1}{2}
\begin{pmatrix}
	0 & 0 & 1 & 0\\
	0 & 0 & 0 & 1\\
	1 & 0 & 0 & 0\\
	0 & 1 & 0 & 0
\end{pmatrix}
\quad.
\end{flalign}
\end{subequations}
Using \eqref{eqn:Rzdz}, \eqref{eqn:RijRji} and  \eqref{eqn:Rmatrix}, one easily checks that
$g \in \Omega^{1}_{A} \otimes_{A} \Omega^{1}_{A}$ is a central element,
hence it defines an $A$-bimodule map $g: A\to\Omega^1_A\otimes_A\Omega^1_A$, see Remark \ref{rem:metricindices}.
The inverse metric $g^{-1}:\Omega^{1}_{A} \otimes_{A} \Omega^{1}_{A} \to A$ is defined on 
the basis $\{\dd z^{i} \otimes_{A} \dd z^{j}\, :\, i,j =  1, \dots ,4\}$ of 
$\Omega^{1}_{A} \otimes_{A} \Omega^{1}_{A}$ by
\begin{subequations}\label{eqn:R4invmetric}
\begin{flalign}
	g^{-1}\big(\dd z^{i} \otimes_{A} \dd z^{j} \big) = g^{ij}\quad,
\end{flalign}
where $g^{ij}$ are the entries of the matrix
\begin{flalign}
(g^{ij}) \,=\, 2
\begin{pmatrix}
	0 & 0 & 1 & 0\\
	0 & 0 & 0 & 1\\
	1 & 0 & 0 & 0\\
	0 & 1 & 0 & 0
\end{pmatrix}
\quad.
\end{flalign}
\end{subequations}
Observe that
\begin{flalign}\label{eqn:R4delta}
\sum_{j = 1}^{4}g_{ij}\, g^{jk} = \delta_{i}^{k}\quad. 
\end{flalign}
\begin{lem}\label{lem:R4metricpair}
The element $g \in \Omega^{1}_{A} \otimes_{A} \Omega^{1}_{A}$ in 
\eqref{eqn:R4metric} defines a (generalized) metric with inverse metric 
$g^{-1}:\Omega^{1}_{A} \otimes_{A} \Omega^{1}_{A} \to A$
defined by \eqref{eqn:R4invmetric}.
\end{lem}
\begin{proof}
It is sufficient to verify the conditions in Remark \ref{rem:metricindices}
on the basis $\{\dd z^{k} \in \Omega^{1}_{A}\}$. Using \eqref{eqn:R4delta}, we compute
\begin{flalign}
	\sum_{i,j = 1}^{4}g_{ij}\, \dd z^{i}\, g^{-1}(\dd z^{j}\otimes_{A}\dd z^{k}) 
	= \sum_{i,j= 1}^{4}g_{ij}\, \dd z^{i}\, g^{jk}
	= \dd z^{k}\quad.
\end{flalign}
The second condition in Remark \ref{rem:metricindices} is confirmed 
through a similar calculation.
\end{proof}

We define a connection $\nabla : \Omega^{1}_{A} \to \Omega^{1}_{A} \otimes_{A} \Omega^{1}_{A}$ 
on $\Omega^1_A$ by
\begin{flalign}\label{eqn:R4StandConn}
\nabla(\dd z^{i}) \,:=\, 0\quad
\end{flalign}
and the left Leibniz rule. Furthermore, we define
an $A$-bimodule isomorphism
$\sigma : \Omega^{1}_{A} \otimes_{A} \Omega^{1}_{A} \to 
\Omega^{1}_{A} \otimes_{A} \Omega^{1}_{A}$ by
\begin{flalign}\label{eqn:R4flip}
\sigma (\dd z^{i} \otimes_{A} \dd z^{j}) 
\,:=\, R^{ji}\, \dd z^{j} \otimes_{A} \dd z^{i} 
\end{flalign}
and left $A$-linear extension to all of $\Omega^{1}_{A} \otimes_{A} \Omega^{1}_{A}$.
(Note that this is analogous to the commutation relations in \eqref{eqn:R4algebra}.)
\begin{lem}\label{lem:R4bimodconn}
The pair $(\nabla,\sigma)$ introduced in \eqref{eqn:R4StandConn} and \eqref{eqn:R4flip} 
defines a bimodule connection.
\end{lem}
\begin{proof}
It remains to confirm the right Leibniz rule from Definition \ref{def:Connections} (ii). 
For this it is sufficient to consider homogeneous elements
$a=z^{j_{1}} \cdots z^{j_{n}} \in A$, for some $n \in \bbZ_{\geq 0}$. We compute
\begin{flalign}
\nn  \nabla\big(\dd z^{i}\, z^{j_{1}}\cdots z^{j_{n}}\big) 
&= \nabla\big(R^{j_{1}i} \cdots R^{j_{n}i}\, z^{j_{1}} \cdots z^{j_{n}}\, \dd z^{i}\big)\\
\nn & =  R^{j_{1}i}\cdots R^{j_{n}i}\, \dd (z^{j_{1}} \cdots z^{j_{n}}) \otimes_{A} \dd z^{i}\\
  &= \sigma\big(\dd z^{i} \otimes_{A} \dd (z^{j_{1}}\cdots z^{j_{n}}) \big)\quad,
\end{flalign}
where in the first equality we used \eqref{eqn:Rzdz} and in the second equality we used the 
left Leibniz rule for the connection \eqref{eqn:R4StandConn}. The last equality follows
by using the Leibniz rule to write $\dd (z^{j_{1}}\cdots z^{j_{n}})
=\sum_{k=1}^n  z^{j_{1}}\cdots z^{j_{k-1}}\, \dd z^{j_k}\, z^{j_{k+1}}\cdots z^{j_{n}}$ and then using
\eqref{eqn:Rzdz}, \eqref{eqn:RijRji} and the definition of $\sigma $ in \eqref{eqn:R4flip}
in order to rearrange these terms.
\end{proof}
\begin{propo}\label{prop:R4Riemstr}
The pair $(g,(\nabla,\sigma))$ defined above is a 
Riemannian structure on $(\Omega^{1}_{A},\dd)$.
\end{propo}
\begin{proof}
We have to verify the two properties of Definition \ref{def:Riemannian}. 
The symmetry property \eqref{eqn:symmetry} is immediate from the definition of 
$g^{-1}$ in \eqref{eqn:R4invmetric}, $\sigma$ in \eqref{eqn:R4flip} and 
$R$ in \eqref{eqn:Rmatrix}. The metric compatibility property \eqref{eqn:metriccompatibility} 
also holds since
\begin{flalign}
	(\id \otimes_{A} g^{-1})\nabla^{\otimes}\big(\dd z^{i}\otimes_{A}\dd z^{j}\big) 
	= 0 
	= \dd \big( g^{-1}(\dd z^{i}\otimes_{A}\dd z^{j}) \big) \otimes_{A}1\quad,
\end{flalign}
where the first equality follows from \eqref{eqn:R4StandConn} 
and the second equality from $g^{-1}(\dd z^{i}\otimes_{A}\dd z^{j}) \in \bbC$.
\end{proof}

The final step is to introduce a spinorial structure.
For the spinor module, we take the $4$-dimensional free left $A$-module
\begin{flalign}\label{eqn:R4spinor}
\EEE \, :=\, A^{4} \,\in\,{}_A\Mod \quad.
\end{flalign}
We denote by $\{e_\alpha \in\EEE\,:\,\alpha=1,\dots,4\}$ the standard basis
for $A^4$, i.e.\ $e_{\alpha}$ is the vector with $1$ in the entry $\alpha$ and $0$ elsewhere.
We define a spin connection $\nabla^{\sp} : \EEE \to \Omega^1_A \otimes_{A} \EEE$ by
\begin{flalign}\label{eqn:R4spinconn}
	\nabla^{\sp}(e_{\alpha})\,:=\, 0
\end{flalign}
and the left Leibniz rule. Introducing a Clifford multiplication
requires some preparations. First, let us recall that the standard
Euclidean gamma matrices in Cartesian coordinates on $\bbR^4$
can be expressed in terms of the $2\times 2$ identity matrix $I_2$
and the three Pauli matrices
\begin{flalign}
	\sigma^{1} = 
	\begin{pmatrix}
	0 & 1\\
	1 & 0
	\end{pmatrix}\quad,\quad
	\sigma^{2} = 
	\begin{pmatrix}
	0 & -i\\
	i & 0
	\end{pmatrix}\quad,\quad
	\sigma^{3} = 
	\begin{pmatrix}
	1 & 0\\
	0 & -1
	\end{pmatrix}
	\quad.
\end{flalign}
Transforming the standard gamma matrices from Cartesian coordinates to our complex
coordinates $z^i$, we obtain
\begin{flalign}
\nn	\gamma^{1} &=
	\begin{pmatrix}
	0 & -\sigma^{1} - i\,\sigma^{2}\\
	\sigma^{1} + i\,\sigma^{2} & 0
	\end{pmatrix} 
	\quad,\quad
	\gamma^{2} =
	\begin{pmatrix}
	0 & -\sigma^{3} - I_{2}\\
	\sigma^{3} -I_{2} & 0
	\end{pmatrix} \quad, \\
	\gamma^{3} &=
	\begin{pmatrix}
	0 & -\sigma^{1} + i\,\sigma^{2}\\
	\sigma^{1} - i\,\sigma^{2} & 0
	\end{pmatrix}
	\quad,\quad
	\gamma^{4} =
	\begin{pmatrix}
	0 & -\sigma^{3} + I_{2}\\
	\sigma^{3} + I_{2} & 0
	\end{pmatrix}\quad. \label{eqn:R4gammamat}
\end{flalign}
By construction, these gamma matrices satisfy the Clifford relations
$\{\gamma^{i},\gamma^{j}\} := \gamma^{i}\,\gamma^{j} + \gamma^{j}\,\gamma^{i} 
= -2\,g^{ij}\,I_{4}$, with $g^{ij}$ given in \eqref{eqn:R4invmetric}. 
These gamma matrices are however {\em not} directly applicable
to our noncommutative space $\bbR^4_\theta$, because the noncommutative
Clifford relations \eqref{eqn:cliffordrels} are given by an anticommutator involving the isomorphism 
$\sigma$ in \eqref{eqn:R4flip}. To address this issue, we introduce the deformed
gamma matrices
\begin{flalign}
\nn	\gamma_\theta^{1} &=
	\begin{pmatrix}
	0 & e^{\frac{i}{4}\theta}\, (-\sigma^{1} - i\,\sigma^{2})\\
	e^{-\frac{i}{4}\theta}\, (\sigma^{1} + i\,\sigma^{2}) & 0
	\end{pmatrix} 
	\quad,\quad
	\gamma_\theta^{2} =
	\begin{pmatrix}
	0 & e^{-\frac{i}{4}\theta}\,(-\sigma^{3} - I_{2})\\
	e^{\frac{i}{4}\theta}\,(\sigma^{3} -I_{2}) & 0
	\end{pmatrix} \quad, \\
	\gamma_\theta^{3} &=
	\begin{pmatrix}
	0 & e^{\frac{i}{4}\theta}\, (-\sigma^{1} + i\,\sigma^{2})\\
	e^{-\frac{i}{4}\theta}\,(\sigma^{1} - i\,\sigma^{2}) & 0
	\end{pmatrix}
	\quad,\quad
	\gamma_\theta^{4} =
	\begin{pmatrix}
	0 & e^{-\frac{i}{4}\theta}\, (-\sigma^{3} + I_{2})\\
	e^{\frac{i}{4}\theta}\, (\sigma^{3} + I_{2}) & 0
	\end{pmatrix}\quad, \label{eqn:R4NCgammamat}
\end{flalign}
which can be obtained from the cocycle deformation techniques
developed in \cite{Brain,AschieriSchenkel,BSShom}.
We define the associated Clifford multiplication
$\gamma : \Omega^{1}_{A}\otimes_A\EEE \to \EEE$  by
\begin{flalign}\label{eqn:R4cliffmult}
	\gamma(\dd z^{i}\otimes_{A} e_{\alpha}) \,:=\, \gamma_{\theta}^{i}\, e_{\alpha}
\end{flalign}
and left $A$-linear extension to all of $\Omega^{1}_{A}\otimes_A\EEE $,
where $\gamma_{\theta}^{i}\, e_{\alpha}$ denotes the action of the matrix 
$\gamma_{\theta}^{i}$ on the basis spinors $e_\alpha\in \EEE= A^4$.
Let us record some useful identities that we will use later.
\begin{lem}\label{lem:R4rcomm}
Define the $\theta$-anticommutator $\{\gamma_{\theta}^{i} , \gamma_{\theta}^{j}\}^{}_{\theta} 
:= \gamma_{\theta}^{i}\,\gamma_{\theta}^{j} + R^{ji}\,\gamma_{\theta}^{j}\,\gamma_{\theta}^{i}$ 
and the $\theta$-commutator $[\gamma_{\theta}^{i} ,  \gamma_{\theta}^{j}]^{}_{\theta} 
:= \gamma_{\theta}^{i}\,\gamma_{\theta}^{j} - R^{ji}\,\gamma_{\theta}^{j}\,\gamma_{\theta}^{i}$. Then
the following properties hold true:
\begin{itemize}
\item[(i)] $\{\gamma_{\theta}^{i} , \gamma_{\theta}^{j}\}_{\theta}^{} = R^{ji}\,\{\gamma_{\theta}^{j} , \gamma_{\theta}^{i}\}_{\theta}^{}$
\item[(ii)] $[\gamma_{\theta}^{i} , \gamma_{\theta}^{j}]_{\theta}^{} = -R^{ji}\,[\gamma_{\theta}^{j} , \gamma_{\theta}^{i}]_{\theta}^{}$
\item[(iii)] $\{\gamma_{\theta}^{i} , \gamma_{\theta}^{j}\}_{\theta}^{} = -2\, g^{ij}\, I_{4}$
\end{itemize}
\end{lem}
\begin{proof}
Items (i) and (ii) follow directly from the definitions and \eqref{eqn:RijRji}.
Item (iii) is a straightforward calculation.
\end{proof}

\begin{propo}
The triple $(\EEE, \nabla^{\sp}, \gamma)$ defined in 
\eqref{eqn:R4spinor}, \eqref{eqn:R4spinconn} and \eqref{eqn:R4cliffmult}
is a spinorial structure on the Riemannian structure $(g,(\nabla,\sigma))$.
\end{propo}
\begin{proof}
We have to verify the two properties of Definition \ref{def:spinorial}. 
The Clifford relations \eqref{eqn:cliffordrels} follow directly from 
Lemma \ref{lem:R4rcomm} (iii), because
\begin{flalign}
\nn (\gamma_{[2]} + \gamma_{[2]} \, (\sigma\otimes_A\id))\big(\dd z^{i} \otimes_{A} \dd z^{j} \otimes_{A} e_{\alpha}\big) 
	&= (\gamma_{\theta}^{i}\,\gamma_{\theta}^{j} + R^{ji}\,\gamma_{\theta}^{j}\,\gamma_{\theta}^{i})\, e_{\alpha} 
	= \{\gamma_{\theta}^{i} , \gamma_{\theta}^{j}\}_{\theta}^{}\, e_{\alpha}\\ 
	&= -2\, g^{ij}\, e_{\alpha} 
	= -2\, g(\dd z^{i} \otimes_{A} \dd z^{j})\, e_{\alpha}\quad.
\end{flalign}
Clifford compatibility \eqref{eqn:cliffordcomp} follows from
\begin{flalign}
(\id\otimes_A\gamma)\nabla^{\otimes}(\dd z^{i} \otimes_{A} e_{\alpha}) = 0= \nabla^{\sp}\gamma(\dd z^{i} \otimes_{A} e_{\alpha})\quad,
\end{flalign}
where we used \eqref{eqn:R4StandConn}, \eqref{eqn:R4spinconn} and \eqref{eqn:R4cliffmult}. 
\end{proof}

We can now compute the Dirac operator 
\eqref{eqn:diracop} associated with our spinorial structure on $\bbR^{4}_{\theta}$.
Expressing elements $s = \sum_{\alpha = 1}^{4}s^{\alpha}\, e_{\alpha} \in \EEE$
in terms of our basis and introducing the notation $\dd a =: \sum_{i=1}^4 \partial_{i}a\, \dd z^{i}$, 
for all $a\in A$, we obtain
\begin{flalign}
\nn	D(s) &= \gamma\big(\nabla^{\sp}(s)\big) = \sum_{\alpha=1}^4\gamma(\dd s^{\alpha} \otimes_{A} e_{\alpha}) 
	= \sum_{\alpha=1}^4\sum_{i=1}^4\partial_{i}s^{\alpha}\, \gamma_{\theta}^{i}\, e_{\alpha} \\
	&=  \sum_{\alpha=1}^4\sum_{i=1}^4 \gamma_{\theta}^{i}\, \partial_{i}s^{\alpha}\, e_{\alpha} 
	= \sum_{i=1}^4\gamma_{\theta}^{i}\, \partial_{i}s\quad,\label{eqn:R4DiracOp}
\end{flalign}
where in the last equality we used the shorthand notation 
$\partial_{i}s := \sum_{\alpha=1}^4\partial_{i}s^{\alpha}\, e_{\alpha}$.

\subsection{\label{subsec:sphere}Noncommutative hypersurface $\bbS^3_\theta \hookrightarrow \bbR^4_\theta$}
We now apply our construction from Section \ref{sec:construction}
to induce the differential, Riemannian and spinorial structure
on $\bbR^4_\theta$ to the noncommutative $3$-sphere
$\bbS^3_\theta \hookrightarrow \bbR^4_\theta$.
This amounts to verifying that this example is a noncommutative hypersurface 
in the sense of Definition \ref{def:hypersurface} and that the Assumptions 
\ref{assu:dftransparent}, \ref{assu:Pitransparent} and \ref{assu:nabladdfcentral} 
for our general construction hold true. We shall also provide
explicit expressions for these induced structures and in particular for
the induced Dirac operator. To simplify our notation, we will suppress
in what follows the square brackets denoting equivalence classes.
It will be clear from the context and our general construction in Section 
\ref{sec:construction} which of the following expressions are considered in quotient spaces.
\sk

The algebra $B=B_{\bbS^3_\theta}$ of the noncommutative Connes-Landi $3$-sphere 
\cite{CL,CDV} is defined as the quotient
\begin{flalign}\label{eqn:S3algebra}
B  \,:=\, A\big/(f)
\end{flalign}
of the algebra $A=A_{\bbR^4_\theta}$ of $\bbR^4_\theta$ 
(cf.\ \eqref{eqn:R4algebra}) by the ideal generated by the unit sphere relation
\begin{flalign}\label{eqn:S3f}
f \,:=\,\frac{1}{2}\Big( \sum_{i,j=1}^4 g_{ij}\, z^{i}\, z^{j}\, - 1 \Big) \,=\,  \frac{1}{2}\Big(z^{1}\, \overline{z^{1}} + z^{2}\, \overline{z^{2}} - 1\Big) \quad,
\end{flalign}
where the prefactor $\tfrac{1}{2}$ is a convenient normalization for 
the generator of the ideal $(f)\subset A$ that is chosen to match the requirements of Example \ref{ex:levelset}.
From the commutation relations given by \eqref{eqn:R4algebra} and \eqref{eqn:Rmatrix},
one checks that $f\in\mathcal{Z}(A)\subseteq A$ is central. 
\begin{propo}\label{prop:S3etaPi}
The $1$-form
\begin{flalign}\label{eqn:S3eta}
\nu \,:=\, \dd f \,=\, \sum_{i,j=1}^4 g_{ij}\, z^{i}\, \dd z^{j} \, \in\, \Omega^1_A
\end{flalign}
is central and normalized. Hence, by Example \ref{ex:levelset},
$B=B_{\bbS^3_\theta}$ is a noncommutative hypersurface 
of $A=A_{\bbR^4_\theta}$  in the sense of Definition \ref{def:hypersurface}.
The projector $\Pi : q_!(\Omega^1_A)\to q_!(\Omega^1_A)$ 
from Proposition \ref{prop:Piproperties} reads explicitly as
\begin{flalign}\label{eqn:S3Pi}
\Pi (\dd z^{i}) \, =\, \dd z^{i} - z^{i}\, \nu\quad.
\end{flalign}
\end{propo}
\begin{proof}
Centrality of $\nu$ is a simple check using \eqref{eqn:Rzdz} and \eqref{eqn:Rmatrix}
and the normalization condition \eqref{eqn:etanormalized} is proven by
\begin{flalign}
g^{-1}(\nu\otimes_A\nu) = \sum_{i,j,k,l=1}^4  g_{ij}\, z^{i}\,g^{jl} \,g_{kl}\,z^k=
\sum_{i,k=1}^4  g_{ik}\, z^{i}\, z^k =1\quad.
\end{flalign}
The explicit expression for the projector is obtained from a short calculation
\begin{flalign}
\Pi(\dd z^i) 
=\dd z^i - g^{-1}\Big(\dd z^i\otimes_A \sum_{j,k=1}^4 g_{jk}\, z^{j}\, \dd z^{k}\Big)\,\nu 
=\dd z^i - \sum_{j,k=1}^4 g^{ik}\, g_{jk}\, z^{j}\, \nu 
= \dd z^i - z^i\,\nu\quad,
\end{flalign}
where in the second step we used  \eqref{eqn:Rzdz}, \eqref{eqn:Rmatrix} and \eqref{eqn:R4invmetric}
in order to write  $\sum_{j,k=1}^4 g_{jk}\, z^{j}\, \dd z^{k} = \sum_{j,k=1}^4 \dd z^{k}\,  g_{jk}\, z^{j}$.
\end{proof}

\begin{propo}
Assumptions \ref{assu:dftransparent} and \ref{assu:Pitransparent} hold true. 
The induced Riemannian structure from Proposition \ref{prop:RiemannianStructure} 
reads explicitly as
\begin{subequations}
\begin{flalign}
	 g_{B} \,=\,\sum_{i,j=1}^4 g_{ij}\, \dd z^{i} \otimes_{B} \dd z^{j} \,\in\,\Omega^1_B\otimes_B\Omega^1_B\quad,\label{eqn:S3metric}
\end{flalign}
\begin{flalign}
	g_{B}^{-1}\big(\dd z^{i} \otimes_{B} \dd z^{j}\big) \,=\, g^{ij} - z^{i}\, z^{j}\quad, \label{eqn:S3invmetric}
\end{flalign}
\begin{flalign}
	\nabla_{B} (\dd z^{i})  \,=\, -z^i \,\sum_{k,l=1}^4 g_{kl}\,\dd z^k\otimes_B \dd z^l\quad, \label{eqn:S3connection}
\end{flalign}
\begin{flalign}
\sigma_B\big(\dd z^i\otimes_B\dd z^j\big)\,=\, R^{ji}\,\dd z^j\otimes_B \dd z^i\quad. \label{eqn:S3flip}
\end{flalign}
\end{subequations}
\end{propo}
\begin{proof}
Verifying Assumption \ref{assu:dftransparent} is a simple check using \eqref{eqn:Rzdz}, \eqref{eqn:R4flip} and \eqref{eqn:Rmatrix}.
To prove commutativity of the top diagram in Assumption \ref{assu:Pitransparent}, we use \eqref{eqn:S3Pi} and compute
\begin{flalign}
\nn	\sigma\big(\Pi(\dd z^{i}) \otimes_{A} \dd z^{j}\big) 
	&= \sigma\big(\dd z^{i} \otimes_{A} \dd z^{j} \big) - \sigma\big( z^{i}\, \nu \otimes_{A} \dd z^{j}\big)\\
\nn	&= R^{ji}\, \dd z^{j} \otimes_{A} \dd z^{i} -  R^{ji}\, \dd z^{j} \otimes_A z^{i}\, \nu\\
\nn	&= R^{ji}\, \dd z^{j} \otimes_{A} \Pi(\dd z^{i})\\
	&= (\id \otimes_{A} \Pi) \, \sigma\big(\dd z^{i} \otimes_{A} \dd z^{j}\big)\quad,
\end{flalign}
where in the second step we used also \eqref{eqn:dftransparent} and \eqref{eqn:Rzdz}. 
Commutativity of the bottom diagram in Assumption \ref{assu:Pitransparent} is proven by a similar calculation.
\sk

Concerning the explicit expressions for the induced Riemannian structure, 
we observe that \eqref{eqn:S3metric} follows trivially from \eqref{eqn:metricB}.
Equation \eqref{eqn:S3invmetric} follows from \eqref{eqn:inversemetricB},
\eqref{eqn:S3Pi} and a straightforward calculation. Equation \eqref{eqn:S3connection} 
follows from \eqref{eqn:nablaBexplicit} and \eqref{eqn:R4StandConn} by a short calculation
\begin{flalign}
\nabla_B (\dd z^i)= \nabla (\dd z^i )-g^{-1}\big(\dd z^i \otimes_A \nu \big)\,\nabla (\nu)
= -z^i\,\nabla(\nu) = -z^i\,\sum_{k,l=1}^4 g_{kl}\, \dd z^k\otimes_B \dd z^l \quad,
\end{flalign}
where in the last step we used that
\begin{flalign}\label{eqn:S3nablaeta}
	\nabla (\nu) = \sum_{k,l=1}^4 g_{kl} \, \nabla \big(z^k\,\dd z^l\big)
	= \sum_{k,l=1}^4 g_{kl}\, \dd z^k\otimes_A \dd z^l 
\end{flalign}
via the left Leibniz rule and \eqref{eqn:R4StandConn}. Finally, \eqref{eqn:S3flip} 
follows trivially from \eqref{eqn:sigmaB} and \eqref{eqn:R4flip}.
\end{proof}

\begin{propo}
Assumption \ref{assu:nabladdfcentral} holds true. The induced spinorial structure 
from Proposition \ref{prop:SpinStr} reads explicitly as
\begin{subequations}
\begin{flalign}\label{eqn:S3spinor}
	\EEE_{B} \,=\, \frac{\EEE}{f\, \EEE}\quad,
\end{flalign}
\begin{flalign}\label{eqn:S3gamma}
	\gamma_{B}\big(\dd z^{i} \otimes_{B} e_{\alpha}\big) 
	\,=\, - \Big(\sum_{k,l=1}^4 g_{kl}\, z^{k}\, \gamma_{\theta}^{l}\, \gamma_{\theta}^{i} + z^{i} \Big) \,e_{\alpha}\quad,
\end{flalign}
\begin{flalign}\label{eqn:S3SpinConn}
	\nabla^{\sp}_{B} (e_{\alpha}) 
	\,=\, \frac{1}{2}\, \sum_{i,j,k,l=1}^4 g_{ij}\, g_{kl}\, z^{k}\, \dd z^{i} \otimes_{B} \gamma_{\theta}^{j}\, \gamma_{\theta}^{l}\, e_{\alpha}\quad.
\end{flalign}
\end{subequations}
\end{propo}
\begin{proof}
Recalling \eqref{eqn:S3nablaeta}, Assumption \ref{assu:nabladdfcentral} is verified 
by a similar calculation as the one that proves centrality of the metric $g$.
\sk

Concerning the explicit expressions for the induced spinorial structure, 
we observe that \eqref{eqn:S3spinor} is just the definition in \eqref{eqn:EEEB}.
Equation \eqref{eqn:S3gamma} follows from \eqref{eqn:gammaB}
by a short calculation
\begin{flalign}
\nn \gamma_B\big(\dd z^i\otimes_B e_\alpha \big)
&=-\gamma_{[2]}\big(\nu\otimes_A\Pi(\dd z^i)\otimes_A e_\alpha\big)\\
\nn &= -\gamma_{[2]}\big(\nu\otimes_A \dd z^i\otimes_A e_\alpha\big)
+g^{-1}\big(\dd z^i\otimes_A \nu\big)\, \gamma_{[2]}\big(\nu\otimes_A\nu\otimes_A e_\alpha\big)\\
&= - \Big(\sum_{k,l=1}^4 g_{kl}\, z^{k}\, \gamma_{\theta}^{l}\, \gamma_{\theta}^{i} + z^{i} \Big) \,e_{\alpha}\quad,
\end{flalign}
where in the first step we used \eqref{eqn:cliffordetarels1}
and in the third step we used \eqref{eqn:cliffordetarels2}.
Finally, equation \eqref{eqn:S3SpinConn} follows from writing out \eqref{eqn:nablaspB}
and using  \eqref{eqn:R4spinconn} and \eqref{eqn:S3nablaeta}.
\end{proof}

We now have all the building blocks for computing the induced Dirac operator on $\bbS^{3}_{\theta}$.
\begin{propo}\label{prop:DiracS3}
The induced Dirac operator \eqref{eqn:induceddiracop} on $\bbS^{3}_{\theta}$ is given by
\begin{flalign}\label{eqn:DiracS3}
D_{B}(s) \,=\, -\frac{1}{2}\sum_{i,j=1}^4 [\gamma_{\theta}^{j},\gamma_{\theta}^{i}]_\theta^{}\, \partial_i s\, z_j - \frac{3}{2}\, s\quad,
\end{flalign}
where $z_{i} := \sum_{k=1}^4g_{ik}\, z^{k}$, $\partial_{i}s :=\sum_{\alpha=1}^4 \partial_{i}s^{\alpha}\, e_{\alpha}$
and $[\gamma_{\theta}^{j},\gamma_{\theta}^{i}]_\theta^{}$ is the $\theta$-commutator from Lemma \ref{lem:R4rcomm}.
\end{propo}
\begin{proof}
We have to compute the induced Dirac operator from Proposition \ref{prop:induceddiracop} for our example.
Using \eqref{eqn:R4StandConn} and \eqref{eqn:S3eta}, the first term of \eqref{eqn:induceddiracopexplicit} is given by
\begin{flalign}
\text{\eqref{eqn:induceddiracopexplicit}}{}^{\mathrm{1st}}  = 
-\frac{1}{2}\sum_{\alpha=1}^4\sum_{i,j,k=1}^4 \partial_i s^\alpha\, g_{kj} \,z^k\, \big(\gamma_{\theta}^{j}\,\gamma_{\theta}^{i} - R^{ij} \,
\gamma_{\theta}^{i}\,\gamma_{\theta}^{j}\big)\,e_\alpha
=-\frac{1}{2}\sum_{i,j=1}^4 [\gamma_{\theta}^{j},\gamma_{\theta}^{i}]_\theta^{}\, \partial_i s\, z_j\quad,
\end{flalign}
which yields the first term of \eqref{eqn:DiracS3}. To compute the second term of \eqref{eqn:induceddiracopexplicit},
we first observe that 
\begin{flalign}
(\Pi\otimes_A \id)\nabla(\nu) =\sum_{i,j=1}^4 g_{ij}\,\Pi\big(\dd z^i\big) \otimes_A\dd z^j 
= \sum_{i,j=1}^4 \big(g_{ij}-z_j\,z_i\big)\,\dd z^i\otimes_A\dd z^j\quad,
\end{flalign}
where in the first step we used \eqref{eqn:S3nablaeta} and in the second step \eqref{eqn:S3Pi}.
This element is invariant under applying $\sigma$, i.e.\ $\sigma (\Pi\otimes_A \id)\nabla(\nu) =
(\Pi\otimes_A \id)\nabla(\nu) $, hence we can write
\begin{flalign}
(\Pi\otimes_A \id)\nabla(\nu) = \frac{1}{2}\Big((\Pi\otimes_A \id)\nabla(\nu) + \sigma (\Pi\otimes_A \id)\nabla(\nu)\Big)
\end{flalign}
in the  second term of \eqref{eqn:induceddiracopexplicit}. Using the Clifford relations \eqref{eqn:cliffordrels},
we obtain
\begin{flalign}
\text{\eqref{eqn:induceddiracopexplicit}}{}^{\mathrm{2nd}}  = -\frac{1}{2}\sum_{i,j=1}^4 \big(g_{ij}-z_j\,z_i\big)\, g^{ij}\,s
= -\frac{1}{2}\big(4-1\big)\,s = -\frac{3}{2}\,s\quad,
\end{flalign}
where in the second step we used $\sum_{i,j=1}^4 g_{ij}\,g^{ij} = \sum_{i=1}^4\delta_i^{i} =4$
(cf.\ \eqref{eqn:R4delta}) and the sphere relation $\sum_{i,j=1}^4 z_j\, z_i\, g^{ij} = \sum_{i,j=1}^4 g_{ij}\,z^i\,z^j =1$
(cf.\ \eqref{eqn:S3f}).
\end{proof}

\begin{rem}
For vanishing deformation parameter $\theta=0$, 
our Dirac operator \eqref{eqn:DiracS3} on $\bbS^3_\theta$
reduces to the usual Dirac operator on the commutative $3$-sphere 
$\bbS^3\subseteq \bbR^4$, see e.g.\ \cite[Section 7.1]{Trautman} or \cite{Trautman93}.
\end{rem}

We shall now compare our noncommutative hypersurface Dirac operator \eqref{eqn:DiracS3} on $\bbS^3_\theta$
to the Connes-Landi Dirac operator \cite{CL,CDV} that is obtained
from an isospectral deformation \cite{Brain}. This requires some preliminaries
on Hopf algebras, their coactions and $2$-cocycle deformations, for which we follow
the notations and conventions of \cite{BSSmapping}. Let $H = \mathcal{O}(\bbT^2)$
denote the Hopf algebra of functions on the $2$-torus. A vector space basis
for $H$ is given by $\{t_{(n_1,n_2)}\, :\, (n_1,n_2)\in\bbZ^2\}$, where $t_{(n_1,n_2)}$
denotes the exponential function $e^{i\, (n_1 \,\phi_1 + n_2\,\phi_2)}$ 
with momentum $(n_1,n_2)$. Consider now the left $H$-coaction
$\rho : A_{\bbR^4}\to H\otimes A_{\bbR^4}$ of the torus Hopf algebra on 
the {\em commutative} algebra of functions on $\bbR^4$ that is 
given in complex coordinates by
\begin{flalign}
\rho(z^1) = t_{(2,0)}\otimes z^1 ~,~~\rho(z^2) = t_{(0,2)}\otimes z^2
~,~~\rho(z^3) = t_{(-2,0)}\otimes z^3~,~~\rho(z^4) = t_{(0,-2)}\otimes z^4\quad.
\end{flalign}
When expressed in terms of the Cartesian coordinates $x^\mu$, it is easy to see that this describes (double covers of) 
rotations in the $(x^1,x^2)$-plane and rotations in the $(x^3,x^4)$-plane. 
The noncommutative algebra $A_{\bbR^4_\theta}$ 
given in \eqref{eqn:R4algebra} can be obtained as a deformation quantization
of the commutative algebra $A_{\bbR^4}$ by introducing the star-product
\begin{flalign}
a\star_\theta a^\prime \,:= \,\sigma_{\theta}\big(a_{(-1)}^{}\otimes a^\prime_{(-1)}\big)~a_{(0)}^{}\,a^\prime_{(0)}\quad,
\end{flalign}
where we used the standard Sweedler notation $\rho (a) = a_{(-1)}\otimes a_{(0)}$ for left coactions.
The $2$-cocycle $\sigma_{\theta} : H\otimes H \to \bbC$ is defined by
\begin{flalign}
\sigma_{\theta}\big(t_{(n_1,n_2)}\otimes t_{(m_1,m_2)}\big)\, :=\, \exp\Big(\tfrac{i \theta}{4} \big(n_1\,m_2 - n_2\,m_1\big)\Big)\quad.
\end{flalign}
Similarly, the algebra $B = B_{\bbS^3_\theta} $ of the Connes-Landi sphere \eqref{eqn:S3algebra}
can be obtained as a deformation quantization of the algebra $B_{\bbS^3}$ of the commutative $3$-sphere.
\sk

One can also obtain the module of noncommutative spinors $\EEE$ on $\bbR^4_\theta$ 
in \eqref{eqn:R4spinor} as a deformation quantization of
the module of commutative spinors by introducing the left $H$-coaction
\begin{flalign}
\rho(e_1) = t_{(1,1)}\otimes e_1~,~~
\rho(e_2) = t_{(-1,-1)}\otimes e_2~,~~
\rho(e_3) = t_{(1,-1)}\otimes e_3~,~~
\rho(e_4) = t_{(-1,1)}\otimes e_4
\end{flalign}
and the associated star-module structure
\begin{flalign}
a \star_\theta s \, :=\, \sigma_{\theta}\big(a_{(-1)}^{}\otimes s_{(-1)}\big)~a_{(0)}^{}\,s_{(0)}\quad.
\end{flalign}
The same is true for the spinor module $\EEE_{B}$ on $\bbS^3_\theta$ given in \eqref{eqn:S3spinor}.
When expressed in terms of these star-products, our Dirac operator \eqref{eqn:DiracS3} on $\bbS^3_\theta$
reads as
\begin{flalign}\label{eqn:DiracS3star}
D_{B} (s) \,=\, -\frac{1}{2}\sum_{i,j=1}^4 [\gamma_{\theta}^{j},\gamma_{\theta}^{i}]_\theta^{}\, \partial^\theta_i s \star_\theta z_j 
- \frac{3}{2}\, s\quad,
\end{flalign}
where $\partial^\theta_i $ is defined by $\dd a = \partial_i^\theta a \star_\theta \dd z^i$
with respect to the deformed module structure.
\sk

The Connes-Landi Dirac operator $D_{\mathrm{CL}}$ on $\bbS^3_\theta$ is given by regarding
the classical Dirac operator on $\bbS^3$ as an operator on the 
deformed spinor module,  see \cite{CL,CDV,Brain} for details. 
Concretely, it is given by setting the deformation parameter $\theta=0$ in \eqref{eqn:DiracS3star}, i.e.\
\begin{flalign}\label{eqn:DCL}
D_{\mathrm{CL}} (s) \,=\, -\frac{1}{2}\sum_{i,j=1}^4 [\gamma^{j},\gamma^{i}]\, \partial_i s \, z_j 
- \frac{3}{2}\, s\quad,
\end{flalign}
where $\partial_i $ is defined by $\dd a = \partial_i a \, \dd z^i$ with respect to the undeformed
module structure.
Because $D_{\mathrm{CL}}$ is equivariant under the torus action,
it satisfies the following property
\begin{flalign}\label{eqn:DCLproperty}
D_{\mathrm{CL}}(a\star_\theta s)\,=\, a\star_\theta D_{\mathrm{CL}} (s) +  
\gamma_{\bbS^3}\big(\dd a \otimes_{B_{\bbS^3_\theta}} s \big)\quad,
\end{flalign}
where $\omega  \otimes_{B_{\bbS^3_\theta}} s := 
\sigma_{\theta} (\omega_{(-1)}\otimes s_{(-1)} )~ \omega_{(0)} \otimes_{B_{\bbS^3}} s_{(0)} $
denotes the deformed tensor product and $\gamma_{\bbS^3}$ the classical Clifford multiplication.
Observe that this is precisely the same property that our hypersurface Dirac
operator $D_B$ satisfies by Proposition \ref{prop:Diracproperty}. This is because,
in the present context of deformation quantization,
our noncommutative Clifford multiplication \eqref{eqn:R4cliffmult} coincides by construction
with the classical Clifford multiplication regarded as a map on the deformed modules.
(The same is true for the induced Clifford multiplication \eqref{eqn:gammaB}
on the noncommutative hypersurface
$\bbS^3_\theta$ because the normalized form $\nu$ in \eqref{eqn:S3eta} is invariant under the torus action.)
With these preparations, we can now prove the following comparison result.
\begin{propo}\label{propo:DiraccomparisonS3}
The hypersurface Dirac operator \eqref{eqn:DiracS3star} on $\bbS^3_\theta$
coincides with the  Connes-Landi Dirac operator $D_{\mathrm{CL}}$. 
\end{propo}
\begin{proof}
Because both $D_B$ and $D_{\mathrm{CL}}$ satisfy the same property \eqref{eqn:DCLproperty},
they coincide if and only if $D_B(e_\alpha) = D_{\mathrm{CL}}(e_\alpha)$, for all basis spinors $e_\alpha$.
The latter follows from \eqref{eqn:DiracS3star} and \eqref{eqn:DCL}
because $D_B(e_\alpha) = -\tfrac{3}{2} e_\alpha = D_{\mathrm{CL}}(e_\alpha)$.
\end{proof}

\subsection{\label{subsec:torus}Noncommutative hypersurface $\bbT^2_\theta \hookrightarrow \bbS^3_\theta$}
In this section we apply our construction from Section \ref{sec:construction}
to induce the differential, Riemannian and spinorial structure
on $\bbS^3_\theta$ (cf.\ Section \ref{subsec:sphere}) to the noncommutative $2$-torus
$\bbT^2_\theta \hookrightarrow \bbS^3_\theta$.
In analogy to Section  \ref{subsec:sphere}, this amounts to verifying that this example is a noncommutative hypersurface 
in the sense of Definition \ref{def:hypersurface} and that the Assumptions 
\ref{assu:dftransparent}, \ref{assu:Pitransparent} and \ref{assu:nabladdfcentral} 
for our general construction hold true. 
We shall also provide
explicit expressions for these induced structures and in particular for
the induced Dirac operator. We will again suppress
in what follows the square brackets denoting equivalence classes in order to simplify our notations.
\sk

Consider the quotient
\begin{flalign}\label{eqn:T2algebra}
C\,:=\, B\big/(\widetilde{f})
\end{flalign}
of the algebra $B=B_{\bbS^3_{\theta}}$ of $\bbS^3_\theta$ (cf.\ \eqref{eqn:S3algebra})
by the ideal generated by
\begin{flalign}\label{eqn:T2f}
\widetilde{f}\,:=\, \frac{1}{2}\Big( \sum_{i,j=1}^4 h_{ij} \,z^i\,z^j  \Big)\,=\,\frac{1}{2}\Big( z^1\,\overline{z^1} - z^2\,\overline{z^2}\Big)\quad,
\end{flalign}
where $h_{ij}$ are the entries of the matrix
\begin{flalign}\label{eqn:hmatrix}
(h_{ij})\,:=\,\frac{1}{2}\begin{pmatrix}
0&0&1&0\\
0&0&0&-1\\
1&0&0&0\\
0&-1&0&0
\end{pmatrix}\quad.
\end{flalign}
From the commutation relations given by \eqref{eqn:R4algebra} and \eqref{eqn:Rmatrix},
one checks that $\widetilde{f}\in\mathcal{Z}(B)\subseteq B$ is central. 
We denote the quotient map
\begin{flalign}
\widetilde{q}\,:\, B~\longrightarrow~ C
\end{flalign}
by a tilde in order to distinguish it from the quotient map $q : A\to B$ in Section \ref{subsec:sphere}.
To recognize $C = C_{\bbT^2_\theta}$ as the algebra of the noncommutative $2$-torus $\bbT^2_\theta$, 
let us recall from \eqref{eqn:S3algebra} that $B=A/(f)$, hence $C = A/(f,\widetilde{f})$
is the quotient of the algebra $A=A_{\bbR^4_\theta}$ of $\bbR^4_\theta$ by
the ideal generated by the two relations $f$ and $\widetilde{f}$ in \eqref{eqn:S3f} and \eqref{eqn:T2f}.
The usual torus relations for the rescaled coordinates
$u:=\sqrt{2}\,z^1$ and $v:=\sqrt{2}\,z^2$ are then obtained from the linear combinations
\begin{flalign}\label{eqn:lincombinationrelations}
2\,(f +\widetilde{f}) \,=\, 2\,z^1\,\overline{z^1} -1=u\,\overline{u}-1 \quad,\quad
2\,(f -\widetilde{f}) \,=\, 2\,z^2\,\overline{z^2} -1 = v\,\overline{v}-1\quad.
\end{flalign}
\begin{propo}\label{prop:T2etaPi}
The $1$-form
\begin{flalign}\label{eqn:T2eta}
\widetilde{\nu} \,:=\, \dd \widetilde{f} \,=\, \sum_{i,j=1}^4 h_{ij}\,z^i\,\dd z^j \,\in\, \Omega^1_B
\end{flalign}
is central and normalized. Hence, by Example \ref{ex:levelset},
$C=C_{\bbT^2_\theta}$ is a noncommutative hypersurface 
of $B=B_{\bbS^3_\theta}$  in the sense of Definition \ref{def:hypersurface}.
The projector $\widetilde{\Pi} : \widetilde{q}_!(\Omega^1_B)\to \widetilde{q}_!(\Omega^1_B)$ 
from Proposition \ref{prop:Piproperties} reads explicitly as
\begin{flalign}\label{eqn:T2Pi}
\widetilde{\Pi}(\dd z^i)\,=\, \dd z^i + (-1)^i \, z^i\,\widetilde{\nu}\quad.
\end{flalign}
\end{propo}
\begin{proof}
Centrality of $\widetilde{\nu}$ is a simple check using \eqref{eqn:Rzdz} and \eqref{eqn:Rmatrix}.
To prove the normalization condition, we use \eqref{eqn:S3invmetric} and compute
\begin{flalign}
g_B^{-1}\big(\widetilde{\nu}\otimes_B\widetilde{\nu}\big)= 
\sum_{i,j,k,l=1}^4 h_{ij}\, z^i\,\big(g^{jl} - z^j\, z^l\big)\, h_{kl}\,z^k
= \sum_{i,k=1}^4 g_{ik}\,z^i\,z^k = 1\quad,
\end{flalign}
where in the second step we used \eqref{eqn:T2f} and the identity
\begin{flalign}\label{eqn:ghcompatibility}
\sum_{j,l=1}^4 h_{ij}\,g^{jl}\, h_{kl}=g_{ik}\quad,
\end{flalign}
and in the last step we used \eqref{eqn:S3f}.
The explicit expression for the projector is obtained from a short calculation
\begin{flalign}
\nn \widetilde{\Pi} (\dd z^i ) &= \dd z^i - g_B^{-1}\big(\dd z^i \otimes_B\widetilde{\nu}\big)\,\widetilde{\nu}
=\dd z^i - \sum_{k,l=1}^4 \big(g^{il} - z^i\,z^l\big) \, h_{kl}\, z^k\,\widetilde{\nu}\\
&= \dd z^i - \sum_{k,l=1}^4 g^{il}\, h_{kl}\, z^k\,\widetilde{\nu} = \dd z^i + (-1)^i \, z^i\,\widetilde{\nu}\quad,
\end{flalign}
where in the second step we used \eqref{eqn:S3invmetric},
in the third step we used \eqref{eqn:T2f} and the last step follows from
$\sum_{l=1}^4 g^{il}\, h_{kl} = -(-1)^i\,\delta^i_k$.
\end{proof}

\begin{propo}
Assumptions \ref{assu:dftransparent} and \ref{assu:Pitransparent} hold true. 
The induced Riemannian structure from Proposition \ref{prop:RiemannianStructure} 
reads explicitly as
\begin{subequations}
\begin{flalign}
	 g_{C} \,=\, \sum_{i,j=1}^4 g_{ij}\, \dd z^{i} \otimes_{C} \dd z^{j} \,\in\,\Omega^1_C \otimes_C \Omega^1_C\quad,\label{eqn:T2metric}
\end{flalign}
\begin{flalign}
	g_{C}^{-1}\big(\dd z^{i} \otimes_{C} \dd z^{j}\big) \,=\, g^{ij} -\big(1+ (-1)^i\,(-1)^j\big)\,z^i\,z^j  \quad, \label{eqn:T2invmetric}
\end{flalign}
\begin{flalign}
	\nabla_{C} (\dd z^{i})  \,=\,  - z^i\, \sum_{k,l=1}^4 \big(g_{kl} - (-1)^i\, h_{kl}\big)\,\dd z^k\otimes_C \dd z^l  \quad,\label{eqn:T2connection}
\end{flalign}
\begin{flalign}
\sigma_C\big(\dd z^i\otimes_C\dd z^j\big)\,=\, R^{ji}\,\dd z^j\otimes_C \dd z^i\quad. \label{eqn:T2flip}
\end{flalign}
\end{subequations}
\end{propo}
\begin{proof}
Verifying Assumption \ref{assu:dftransparent} is a simple check using 
\eqref{eqn:Rzdz}, \eqref{eqn:R4flip} and \eqref{eqn:Rmatrix}.
To prove commutativity of the top diagram in Assumption \ref{assu:Pitransparent}, we use 
\eqref{eqn:T2Pi} and compute
\begin{flalign}
\nn	\sigma_B\big(\widetilde{\Pi}(\dd z^{i}) \otimes_{B} \dd z^{j}\big) 
	&= R^{ji}\, \dd z^{j} \otimes_{B} \dd z^{i} + (-1)^i \, z^i\, \dd z^{j} \otimes_B \widetilde{\nu}\\
\nn	&= R^{ji}\, \dd z^{j} \otimes_{B} \dd z^{i} +  R^{ji} \, \dd z^{j} \otimes_B(-1)^i \,z^i\, \widetilde{\nu}\\
	&= (\id \otimes_{B} \widetilde{\Pi} ) \, \sigma_B\big(\dd z^{i} \otimes_{B} \dd z^{j}\big)\quad,
\end{flalign}
where in the second step we used \eqref{eqn:Rzdz}.
Commutativity of the bottom diagram in Assumption \ref{assu:Pitransparent} is proven by a similar calculation.
\sk

We observe that \eqref{eqn:T2metric} follows trivially from \eqref{eqn:metricB}
and \eqref{eqn:T2invmetric} follows from \eqref{eqn:inversemetricB},
\eqref{eqn:T2Pi} and a straightforward calculation. 
Equation \eqref{eqn:T2connection} follows from \eqref{eqn:nablaBexplicit},
\eqref{eqn:S3connection} and
\begin{flalign}\label{eqn:T2nablaeta}
	\nabla_B (\widetilde{\nu})= \sum_{k,l=1}^4 h_{kl}\, \dd z^k\otimes_B \dd z^l 
\end{flalign}
by a short calculation.  Finally, \eqref{eqn:T2flip} 
follows trivially from \eqref{eqn:sigmaB} and \eqref{eqn:S3flip}.
\end{proof}

\begin{propo}
Assumption \ref{assu:nabladdfcentral} holds true. The induced spinorial structure 
from Proposition \ref{prop:SpinStr} reads explicitly as
\begin{subequations}
\begin{flalign}\label{eqn:T2spinor}
	\EEE_{C} \,=\, \frac{\EEE_B}{\widetilde{f}\, \EEE_B} \,=\, \frac{\EEE}{f\EEE \cup \widetilde{f}\EEE}\quad,
\end{flalign}
\begin{flalign}\label{eqn:T2gamma}
	\gamma_{C}\big(\dd z^{i} \otimes_{C} e_{\alpha}\big) 
	\,&=\, \Big(z^i \!\!\!\sum_{k,l,m,n=1}^4  g_{mn}\,z^m\,h_{kl}\,z^k\,\gamma^l_\theta\,\gamma^n_\theta 
	- \sum_{k,l=1}^4 h_{kl}\,z^k\,\gamma^l_\theta \,\gamma^i_\theta  + (-1)^i\,z^i\Big)\,e_\alpha\quad,
\end{flalign}
\begin{flalign}\label{eqn:T2SpinConn}
	\nabla^{\sp}_{C} (e_{\alpha}) 
	\,=\,  \frac{1}{2}\sum_{i,j,k,l=1}^4\Big(g_{kl}\,z^k\,g_{ij}\,\dd z^i + h_{kl}\,z^k\,h_{ij}\,\dd z^i \Big)\otimes_C \gamma^j_\theta\,\gamma^l_\theta\,e_\alpha\quad.
\end{flalign}
\end{subequations}
\end{propo}
\begin{proof}
Recalling \eqref{eqn:T2nablaeta}, Assumption \ref{assu:nabladdfcentral} is verified 
by a similar calculation as the one that proves centrality of $\widetilde{f}$ given in \eqref{eqn:T2f}.
The explicit expressions in \eqref{eqn:T2spinor}, \eqref{eqn:T2gamma} 
and \eqref{eqn:T2SpinConn} follow easily from the definitions (cf.\ \eqref{eqn:EEEB},
\eqref{eqn:gammaB} and \eqref{eqn:nablaspB}) by straightforward calculations.
(To obtain \eqref{eqn:T2SpinConn}, one has to recall that $\widetilde{\nu} = \dd\widetilde{f}=0$ in $\Omega^1_C$.)
\end{proof}

\begin{propo}
The induced Dirac operator \eqref{eqn:induceddiracop} on $\bbT^{2}_{\theta}$ is given by
\begin{flalign}\label{eqn:DiracT2}
D_{C}(s) \,=\, - \frac{1}{2}\sum_{i,j=1}^4[\gamma^j_\theta,\gamma^i_\theta]_\theta^{}\,
\Big(\partial_i s\,\widetilde{z}_j  - \sum_{k=1}^4 \partial_k s\,z^k\,z_i\,\widetilde{z}_j  - s\, z_i\, \widetilde{z}_j\Big)\quad,
\end{flalign}
where $z_{i} := \sum_{k=1}^4g_{ik}\, z^{k}$, $\widetilde{z}_{i} := \sum_{k=1}^4h_{ik}\, z^{k}$,
$\partial_{i}s :=\sum_{\alpha=1}^4 \partial_{i}s^{\alpha}\, e_{\alpha}$ and $[\gamma_{\theta}^{j},\gamma_{\theta}^{i}]_\theta^{}$ 
is the $\theta$-commutator from Lemma \ref{lem:R4rcomm}.
\end{propo}
\begin{proof}
The proof is a straightforward but slightly lengthy calculation and hence will not 
be written out in detail.
\end{proof}

From our presentation given in \eqref{eqn:DiracT2}, it is not easy to
interpret and understand $D_C$ as a Dirac operator on the flat
noncommutative torus $\bbT^2_\theta$. We will now simplify \eqref{eqn:DiracT2} 
to a form that admits an obvious interpretation. For this it will be useful to introduce
the standard generators 
\begin{subequations}
\begin{flalign}
u\,:=\,\sqrt{2}\, z^1\quad ,\quad v\,:=\, \sqrt{2}\,z^2\quad,\quad 
\overline{u}\,:=\,\sqrt{2}\, \overline{z^1}\quad,\quad \overline{v}\,:=\, \sqrt{2}\,\overline{z^2}
\end{flalign}
of the algebra $C$ of $\bbT^2_\theta$, which satisfy the relations
\begin{flalign}
\overline{u}\,u \,=\, 1\quad,\quad \overline{v}\,v \,=\, 1\quad,\quad u\,v = e^{i\theta}\, v\, u\quad.
\end{flalign}
\end{subequations}
The module $\Omega^1_C$ of $1$-forms on $C$ is a $2$-dimensional free module
with central basis
\begin{flalign}\label{eqn:ddphibasis}
\dd\phi^1 \,:=\, \frac{1}{i}\,\overline{u}\, \dd u \quad,\quad \dd\phi^2 \,:=\,\frac{1}{i}\,\overline{v}\,\dd v\quad,
\end{flalign}
where $i\in\bbC$ denotes the imaginary unit. (Our notation is inspired
by thinking of $u = e^{i\,\phi^1}$ and $v= e^{i\,\phi^2}$ as exponential functions.)
The inverse metric \eqref{eqn:T2invmetric} in this basis reads as
\begin{flalign}\label{eqn:ginverseT2explicit}
g_C^{-1}\big(\dd \phi^i\otimes \dd \phi^j \big) \,=\, 2\,\delta^{ij}\quad,
\end{flalign}
where the factor $2$ is due to the fact that our embedded noncommutative torus
$\bbT^2_\theta \hookrightarrow \bbS^3_\theta$ has radius $\frac{1}{\sqrt{2}}$, see \eqref{eqn:lincombinationrelations}.
The differential $\dd a = \partial_{\phi^1} a \,\dd\phi^1 + \partial_{\phi^2}a\,\dd\phi^2$
of any $a\in C$ can be expressed in the basis \eqref{eqn:ddphibasis}. Comparing this
to $\dd a = \sum_{i=1}^4\partial_i a\,\dd z^i \in\Omega^1_C$, we find
\begin{flalign}\label{eqn:partialsonT2}
\partial_1 a\,=\,\frac{2}{i}\,\partial_{\phi^1}a\,\overline{z^1}\quad,\quad
\partial_2 a\,=\, \frac{2}{i}\,\partial_{\phi^2}a\,\overline{z^2}\quad,\quad
\partial_3 a\,=\,0\quad,\quad \partial_4 a\,=\, 0
\end{flalign}
for the noncommutative partial derivatives along $z^i$.
\sk

To simplify the induced Dirac operator \eqref{eqn:DiracT2} on $\bbT^2_\theta$,
we use the Clifford relations in the form of Lemma \ref{lem:R4rcomm} (iii)
and obtain after a short calculation
\begin{flalign}
D_C(s) \,=\, -\gamma\Big(\widetilde{\nu}\otimes_C \sum_{i=1}^4
\Big(\gamma^i_\theta \,\partial_i s - \gamma^i_\theta\, s\,z_i\Big) \Big)
+ \gamma_{[2]}\Big(\widetilde{\nu}\otimes_C\nu \otimes_C \sum_{i=1}^4 \partial_i s\,z^i \Big)
+ \sum_{i=1}^4 (-1)^i\, \partial_i s\,z^i\quad.
\end{flalign}
Applying the map $\gamma(\widetilde{\nu} \otimes_C (-)) : \EEE_C\to\EEE_C$ to this expression, 
which squares to $-\id$ because $\widetilde{\nu}$ is normalized, we define 
\begin{flalign}
\nn \widetilde{D}_C(s) \,:=&\, \gamma\big(\widetilde{\nu}\otimes_C D_C(s)\big)\\
\,=&\, 
 \sum_{i=1}^4 \Big(\gamma^i_\theta \,\partial_i s - \gamma^i_\theta\, s\,z_i\Big)
 - \gamma\Big(\nu \otimes_C \sum_{i=1}^4 \partial_i s\,z^i \Big)
 + \gamma\Big(\widetilde{\nu}\otimes_C \sum_{i=1}^4 (-1)^i\, \partial_i s\,z^i\Big)\quad.
\end{flalign}
Inserting \eqref{eqn:partialsonT2}, \eqref{eqn:S3eta} and \eqref{eqn:T2eta} into this
expression and carrying out all summations, one obtains
\begin{flalign}
\nn \widetilde{D}_C(s)\,&=\,\frac{1}{i}\Big( \gamma^1_\theta\, \partial_{\phi^1}s\,\overline{z^1} -
\gamma^3_\theta\, \partial_{\phi^1}s\,z^1\Big)  - \frac{1}{2}\,\Big(\gamma^1_\theta \,s\,  \overline{z^1}+\gamma^3_\theta \,s\,  z^1\Big)\\
&\qquad~\qquad +  \frac{1}{i}\Big(\gamma^2_\theta\, \partial_{\phi^2}s\,\overline{z^2} - \gamma^4_\theta\, \partial_{\phi^2}s\,z^2 \Big)
- \frac{1}{2}\,\Big(\gamma^2_\theta \,s\,  \overline{z^2} + \gamma^4_\theta \,s\,  z^2  \Big)\quad.\label{eqn:Diractilde2}
\end{flalign}
Let us introduce the $C$-module map $\widetilde{\gamma} : \Omega^1_C\otimes_C\EEE_C\to\EEE_C$ by defining
\begin{flalign}
\widetilde{\gamma}\big(\dd\phi^1\otimes_C s\big)\,:=\,\frac{1}{i} \Big( \gamma^1_\theta \,s\,\overline{z^1} - \gamma^3_\theta\,s\,z^1\Big)
\quad,\quad 
\widetilde{\gamma}\big(\dd\phi^2\otimes_C s\big)\,:=\, \frac{1}{i}\Big(\gamma^2_\theta \,s\,\overline{z^2} - \gamma^4_\theta\,s\,z^2\Big)\quad,
\end{flalign}
for all $s\in\EEE_C$. One easily shows that $\widetilde{\gamma}$ satisfies the Clifford relations
\begin{flalign}
\widetilde{\gamma}_{[2]}\big(\dd\phi^i \otimes_C\dd\phi^j \otimes_C s\big) +
\widetilde{\gamma}_{[2]}\big(\dd\phi^j \otimes_C\dd\phi^i \otimes_C s\big) = -2\,g_C^{-1}\big(\dd \phi^i\otimes \dd \phi^j \big)\,s
=-4\,\delta^{ij}\,s
\end{flalign}
for the inverse metric \eqref{eqn:ginverseT2explicit}. (Note that there is no $\sigma$
in this expression because $\sigma(\dd\phi^i\otimes_C\dd\phi^j) = \dd\phi^j\otimes_C\dd\phi^i$.)
This allows us to write \eqref{eqn:Diractilde2} as
\begin{flalign}
\nn \widetilde{D}_C(s) \,&=\, \widetilde{\gamma}\Big(\dd \phi^1 \otimes_C \Big(\partial_{\phi_1} s 
+\frac{1}{4} \widetilde{\gamma}\big(\dd\phi^1 \otimes_C \big(\gamma^1_\theta \,s\,  \overline{z^1}+\gamma^3_\theta \,s\,  z^1\big) \big)\Big)\Big)\\
\nn &\qquad~\qquad + \widetilde{\gamma}\Big(\dd \phi^2 \otimes_C \Big(\partial_{\phi_2} s 
+\frac{1}{4} \widetilde{\gamma}\big(\dd\phi^2 \otimes_C\big(\gamma^2_\theta \,s\,  \overline{z^2} + \gamma^4_\theta \,s\,  z^2 \big) \big)\Big)\Big)\\
&=\widetilde{\gamma}\Big(\dd \phi^1 \otimes_C \Big(\partial_{\phi_1} s 
+\frac{1}{8i}  [\gamma^1_\theta,\gamma^3_\theta]_\theta^{}\,s\Big)\Big)
+ \widetilde{\gamma}\Big(\dd \phi^2 \otimes_C \Big(\partial_{\phi_2} s 
+\frac{1}{8i} [\gamma^2_\theta,\gamma^4_\theta]_{\theta} \,s \Big)\Big)
\quad,\label{eqn:Diractilde3}
\end{flalign}
which we recognize as the Dirac operator on $\bbT^2_\theta$ 
corresponding to a rotating frame spin structure, see \cite{BarrettGaunt}.
By a direct calculation, one shows that the spectrum of this operator,
and hence the spectrum of the Dirac operator $D_C$ in \eqref{eqn:DiracT2} 
on the noncommutative torus $\bbT^2_\theta$, is given by
\begin{flalign}\label{eqn:T2spectrum}
\Big\{ \pm \sqrt{2}~\sqrt{\left(m+\tfrac{1}{2}\right)^2 + \left(n+\tfrac{1}{2}\right)^2}\,:\, m,n\in \bbZ  \Big\}\quad.
\end{flalign}
We note that this coincides with the spectrum of the Dirac operator corresponding to
the $(1,1)$ spin structure on the commutative $2$-torus $\bbT^2$, see e.g.\ \cite{Friedrich}.
(The factor $\sqrt{2}$ in \eqref{eqn:T2spectrum} is because our noncommutative torus
$\bbT^2_\theta \hookrightarrow \bbS^3_\theta$ has radius $\frac{1}{\sqrt{2}}$.)
\sk

By the same argument as in Proposition \ref{propo:DiraccomparisonS3},
which however involves now a considerably lengthier calculation to compute
$D_{C}(e_\alpha)$ on the basis spinors, one can show that, when expressed in terms of star-products,
our noncommutative hypersurface Dirac operator \eqref{eqn:Diractilde3}
on $\bbT^2_\theta$ coincides with the isospectral deformation \cite{Brain} 
of the classical Dirac operator of type $(1,1)$ on the commutative $2$-torus,
acting as in \cite{BarrettGaunt} on doubled, i.e.\  $4$-dimensional, spinors.


\section*{Acknowledgments}
We would like to thank Branimir {\'C}a{\'c}i{\'c} for very useful comments
on the manuscript and in particular for suggesting Definition \ref{def:hypersurface} to us.
We would also like to thank Joakim Arnlind, John Barrett, James Gaunt, 
Shahn Majid and Axel Tiger Norkvist for useful comments related to this work. 
H.N.\ is supported by a PhD Scholarship from the School of Mathematical
Sciences of the University of Nottingham. 
A.S.\ gratefully acknowledges the financial support of 
the Royal Society (UK) through a Royal Society University 
Research Fellowship (UF150099), a Research Grant (RG160517) 
and two Enhancement Awards (RGF\textbackslash EA\textbackslash 180270 and RGF\textbackslash EA\textbackslash 201051). 


\end{document}